\documentclass{amsart}

\usepackage{amsrefs}
\usepackage{euscript}
\usepackage{amsmath}
\usepackage{amssymb}
\usepackage{enumerate}

\usepackage{amsfonts}
\usepackage{graphicx}
\usepackage[all]{xy}

\numberwithin{equation}{section}

\newcommand{\bC}{{\mathbb C}}
\newcommand{\bP}{{\mathbb P}}
\newcommand{\bR}{{\mathbb R}}
\newcommand{\bZ}{{\mathbb Z}}

\newcommand{\cA}{\mathcal A}
\newcommand{\cB}{\mathcal B}
\newcommand{\cY}{\mathcal Y}
\newcommand{\cP}{\mathcal P}
\newcommand{\cL}{\mathcal L}
\newcommand{\cR}{\mathcal R}
\newcommand{\cH}{\mathcal H}
\newcommand{\sT}{\EuScript T}

\newcommand{\vx}[1][\!]{\vec{x}^{\, #1}}
\newcommand{\vy}[1][\!]{\vec{y}^{\, #1}}

\newcommand{\sJ}{{\EuScript J}}
\newcommand{\sH}{{\EuScript H}}
\newcommand{\Orbit}{{\EuScript O}}
\newcommand{\Chord}{{\EuScript X}}

\newcommand{\bfD}{\mathbf{D}}
\newcommand{\bfk}{\mathbf{k}}

\newcommand{\id}{\operatorname{id}}
\newcommand{\Disct}{\tilde{\mathcal R}}
\newcommand{\Disc}{{\mathcal R}}
\newcommand{\Product}{{\mathcal P}}
\newcommand{\Discbar}{\overline{\Disc}}

\newcommand{\Cylt}{\tilde{\mathcal M}}
\newcommand{\Cyl}{{\mathcal M}}
\newcommand{\Cylbar}{\overline{\Cyl}}
\newcommand{\Ann}{{\mathcal C}}
\newcommand{\Annbar}{\overline{\Ann}}
\newcommand{\Moduli}{\mathcal M}
\newcommand{\OC}{\mathcal {OC}}
\newcommand{\CO}{\mathcal {CO}}
\newcommand{\Wrap}{\EuScript{W}}

\newcommand{\Hom}{\operatorname{Hom}}
\newcommand{\Laction}{\EuScript{A}}
\newcommand{\Ob}{\operatorname{Ob}}
\newcommand{\Univ}{\mathcal{U}}

\renewcommand{\dbar}{\overline{\partial}}
\def\co{\colon\thinspace}

\newtheorem{thm}{Theorem}[section]
\newtheorem{cor}[thm]{Corollary}
\newtheorem{lem}[thm]{Lemma}
\newtheorem{prop}[thm]{Proposition}
\newtheorem{defin}[thm]{Definition}

\theoremstyle{remark}
\newtheorem{rem}[thm]{Remark}

\newcommand{\superscript}[1]{\ensuremath{^{\textrm{#1}}} }

\renewcommand{\th}[0]{\superscript{th}}
\newcommand{\st}[0]{\superscript{st}}

\title[A geometric criterion for generating the Fukaya category]{A geometric criterion for generating the Fukaya category}

\author[M.~Abouzaid]{Mohammed Abouzaid} \date{September 29, 2010}
\thanks{This research was conducted during the period the author served as a Clay Research Fellow. }

\begin{document}
\begin{abstract}
Given a collection of exact Lagrangians in a Liouville manifold, we construct a map from the Hochschild homology of the Fukaya category that they generate to symplectic cohomology.   Whenever the identity in symplectic cohomology lies in the image of this map, we conclude that every Lagrangian lies in the idempotent closure of the chosen collection.  The main new ingredients are (1) the construction of   operations on the Fukaya category controlled by discs with two outputs, and (2) the Cardy relation.
\end{abstract}

\maketitle
\setcounter{tocdepth}{1}
\tableofcontents

\section{Introduction}
Let $M$ be a Liouville manifold, and let $\Wrap$ be a finite collection of objects of its wrapped Fukaya category.  The Lagrangians we allow are exact and  are  modelled after a Legendrian at infinity (see \cite{abouzaid-seidel} and Section \ref{sec:geom-prel}).   In this paper, we define a map
\begin{equation} \label{eq:open_to_closed} H^{*}( \OC)  \co HH_{*}(\Wrap, \Wrap) \to SH^{*}(M) , \end{equation}
where $SH^*(M)$ is symplectic cohomology, and $HH_{*}$ is Hochschild homology.  Symplectic cohomology admits a canonical class called the identity which is the image of the identity in the ordinary cohomology of $M$ under the natural map
\begin{equation*}H^{*}(M)  \to SH^{*}(M)  \end{equation*}
whose importance was stressed by Viterbo in \cite{viterbo}.
\begin{thm}[Generation Criterion] \label{thm:main_thm}
If the identity lies in the image of the composition
\begin{equation} HH_{*}(\cB, \cB) \to  HH_{*}(\Wrap, \Wrap) \to SH^{*}(M)  \end{equation}
for a full subcategory $\cB$, then the objects of $\cB$ split-generate $\Wrap$.
\end{thm}
\begin{rem}
We present a proof of this result using the minimal amount of technology, both  from the point of view of homological algebra and holomorphic curves.  Assuming that the work  \cite{MWW} of M'au, Wehrheim, and Woodward  on holomorphic quilts extends to the wrapped setting (a reasonable, though by no means obvious assumption), one can alternatively prove first that the wrapped Fukaya category, under the assumptions of Theorem \ref{thm:main_thm}, is \emph{homologically smooth} in the sense of Kontsevich (see, e.g. Section 8.1 of \cite{KS}).  Namely, the theory of holomorphic quilts defines a functor from the Fukaya category of the product $M^{2}$ to the category of bimodules over the Fukaya category of $M$; the image of the geometric diagonal under this functor is the diagonal bimodule.  With the help of Appendix \ref{sec:univ-twist-compl}, one can conclude from the assumptions of Theorem \ref{thm:main_thm} that the geometric diagonal lies in the category split-generated by products of objects of $\cB$.  From that, the existence of a functor is sufficient to conclude that the diagonal bimodule over $\Wrap$ is a \emph{perfect bimodule}; i.e. that it lies in the category split-generated by tensor products of left and right Yoneda modules (see Appendix \ref{sec:univ-twist-compl} for details about these modules).  In fact, the construction shows that only objects lying in $\cB$ are required to produce the complex exhibiting the diagonal as a perfect bimodule.  From this one concludes that  the  objects of $\cB$ split-generate $\Wrap$ by a generalisation of Beilinson's argument \cite{beilinson} to the case where morphism spaces are allowed to have infinite rank (the existence of such a generalisation seems known to experts, though it does not appear in the literature in the desired form).
\end{rem}

Combining Theorem \ref{thm:main_thm} with recent work of Bourgeois, Ekholm and Eliahsberg announced in \cite{BEE} should imply that the unstable manifolds of a plurisubharmonic Morse function on a Stein manifold split-generate its wrapped Fukaya category, though there are still technical gaps in relating their symplectic field theory approach with the Floer equations we use.   In the special case of cotangent bundles, the interested reader can find some applications developed in the papers \cites{abouzaid-cotangent,abouzaid-htpy}. In a different direction, the exactness assumption can be dropped, and we expect upcoming joint work with Fukaya, Oh, Ohta, and Ono to address the case of compact (weakly unobstructed) Lagrangians.

For compact unobstructed Lagrangians, one may construct a map analogous to Equation \eqref{eq:open_to_closed} whose target is ordinary cohomology as explained in Section 13 of \cite{FOOO} (see also Conjecture 1 and 2 in \cite{costello} for a discussion from a more abstract topological field theory point of view).   However, because of the lack of Poincar\'e duality in wrapped Floer homology, one cannot directly import intuition from  topological field theory.  We bypass these difficulties by stating Theorem \ref{thm:main_thm} in the concrete manner given, rather than trying to appeal to a more theoretical approach.

To prove the generation criterion, we consider, for any object $K$ in $\Wrap$, the left and right $\Wrap$ modules  $ \cY^{l}_{K}$ and $\cY^{r}_{K}$  given by the Yoneda construction.  Passing to cohomology, these are the modules which respectively assign to a Lagrangian $L$ the wrapped Floer cohomology groups
\begin{equation*}
  HW^{*}(K,L) \textrm{ and }   HW^{*}(L,K)
\end{equation*}
 which are morphism spaces in $\Wrap$.   Using the same notation for the restriction of these modules to $\cB$,  we define a map
\begin{equation*}
\Delta \co  \cB \to  \cY^{l}_{K} \otimes \cY^{r}_{K}
\end{equation*}
of $\cB$ bimodules in Section \ref{sec:coproduct-as-map}; here $\cB$ stands for the diagonal bimodule.  Again, at the level of cohomology, $\Delta$ consists of a collection of maps
\begin{equation*}
 HW^{*}(L,L') \to  HW^{*}(K,L')  \otimes  HW^{*}(L,K)
\end{equation*}
for every pair $(L,L')$ of objects of $\cB$ satisfying appropriate conditions with respect to multiplication by morphism spaces in $\cB$.  Note that composition in the Fukaya category defines a natural map in the other direction. One should think of $\Delta$ as a ``pre-dual'' for multiplication; an honest dual cannot be expected since morphism spaces may be of infinite rank.

One may define Hochschild homology groups with coefficients in any bimodule in such a way that a map of bimodules induces a map on  Hochschild homology.  In the special case where the bimodule is a tensor product $\cL \otimes \cR$ of a left and a right module, we have
\begin{equation*}
HH_{*}(\cB, \cL \otimes \cR) \cong H^{*}(\cR \otimes_{\cB} \cL).
\end{equation*}

Applying this general nonsense to our geometric situation, we conclude that $\Delta$ induces a map
\begin{equation} \label{eq:map_HH_tensor_over_A}
HH_{*}(\Delta) \co HH_{*}(\cB, \cB) \to  H^{*}(\cY^{r}_{K} \otimes_{\cB} \cY^{l}_{K}).
\end{equation}
On the other hand, composition in $\cB$ defines a map 
\begin{equation} \label{eq:multiplication_right_left}
H^{*}(\mu) \co H^{*}(\cY^{r}_{K} \otimes_{\cB} \cY^{l}_{K}) \to HW^{*}(K,K).
\end{equation}
In Section \ref{sec:maps-relating-open}, we shall construct maps  $\OC$ and $\CO$ by counting solutions to a Floer-type equation on a disc with $1$ interior puncture and appropriate boundary conditions, and we shall prove the next result in Section \ref{sec:cardy-relation}:
\begin{prop} \label{prop:cardy_relation}
Up to an overall sign of $(-1)^{\frac{n(n+1)}{2}}$, the following diagram commutes:
\begin{equation}
 \label{eq:cardy_relation}
 \xymatrixcolsep{5pc} \xymatrix{ HH_{*}(\cB, \cB) \ar[r]^{HH_{*}(\Delta)}  \ar[d]^{H^{*}(\OC)} &  H^{*}(\cY^{r}_{K} \otimes_{\cB} \cY^{l}_{K}) \ar[d]^{H^*(\mu)} \\
SH^{*}(M) \ar[r]^{H^{*}(\CO)}  & HW^{*}(K,K)}
\end{equation}
\end{prop}
This a version of the Cardy relation, and requires the study of a moduli space of pseudo-holomorphic maps whose source is an annulus.   For example, if we consider an element in Hochschild homology represented by a endomorphism of $L$ then the composition of $\CO$ and $\OC$ corresponds to the first broken curve in Figure \ref{fig:Cardy}, while we interpret the other composition as the second pair of broken curves.   In this simplest situation, the moduli spaces we use have been recently studied by Biran-Cornea  and Fukaya respectively in the case of monotone and unobstructed compact Lagrangians.   By gluing the interior nodes or the pair of boundary nodes, one obtains an annulus with one boundary circle mapping to $L$ and the other to $K$, and the commutativity of  Diagram \eqref{eq:cardy_relation} follows from generalising the fact that we may interpolate between these two broken curves through the moduli space of annuli. 

\begin{figure}\label{fig:Cardy}
  \centering
\includegraphics{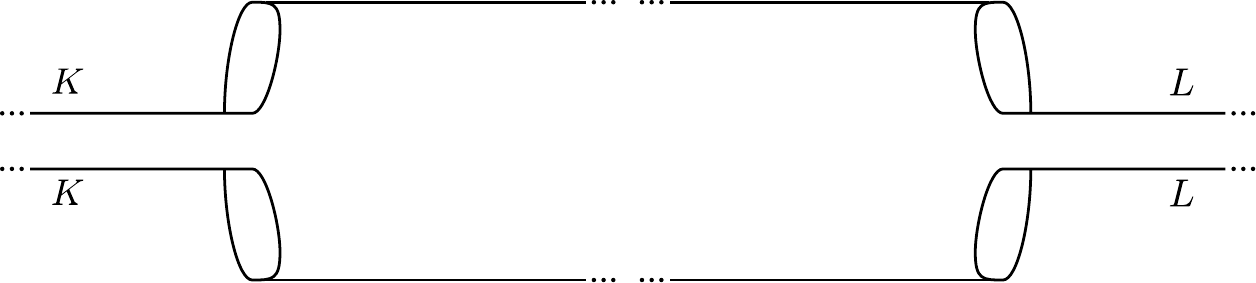} \\
\vspace{.2in}
   \includegraphics{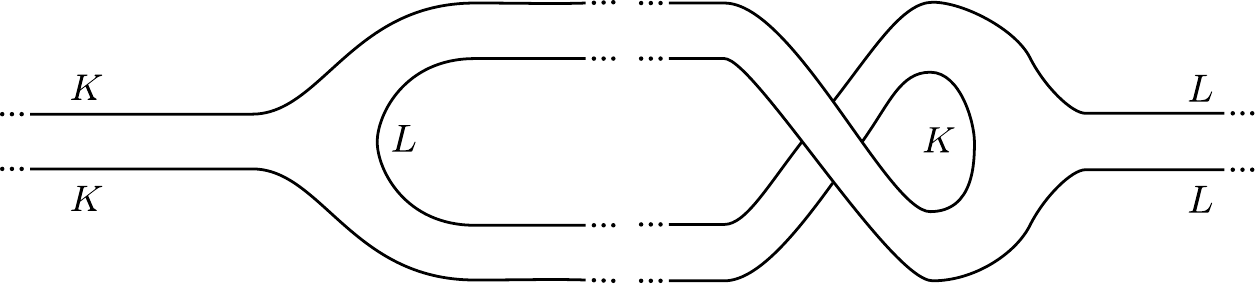} 
 \caption{}
\end{figure}

 We state a final algebraic Lemma proved in Appendix \ref{sec:univ-twist-compl} before establishing the main theorem:
\begin{lem} \label{lem:existence_twisted_complex}
If  the identity of  $ HW^{*}(K,K) $  lies in the image of $H^{*}(\mu)$, then $\cY^{r}_{K}$ is a summand of a twisted complex built from objects of $\cB$.
\end{lem}

\begin{proof}[Proof of Theorem \ref{thm:main_thm}]
The statement that $\cB$ split-generates $\Wrap$ is equivalent to the existence, for each object $K$ of $\Wrap$, of a twisted complex built from objects of $\cB$ that admits $K$ as a summand.  The previous Lemma therefore reduces the proof of split-generation to the verification that the identity of each object lies in the image of   $H^{*}(\mu)$.

The identity  in symplectic cohomology is represented by a count of rigid holomorphic planes. On the other hand, the image of an element of symplectic cohomology under  $H^*(\CO)$ is obtained by counting rigid discs with one interior puncture converging to a representative of this element, and boundary conditions along $K$.   The result of gluing a rigid plane to the interior puncture can be deformed by a standard argument (see, e.g. Section 4.2 of \cite{abouzaid-htpy}) to a rigid disc with boundary on $K$; this count precisely represents the identity of $K$ in its wrapped Floer cohomology group.

Under the hypothesis of Theorem \ref{thm:main_thm}, the commutativity of Diagram \eqref{eq:cardy_relation} therefore implies that the identity in $ HW^{*}(K,K) $  lies in the image of  $H^{*}(\mu)$ which completes the proof.
\end{proof}

\subsubsection*{Acknowledgments}
I am grateful to Kenji Fukaya for showing me his proof that every unobstructed Lagrangian in  $\bC \bP^{n}$ must intersect the Clifford torus, for the discussions we had in which we proved that the Clifford torus split-generates the Fukaya category, and for pointing out that a preliminary version of this paper omitted a necessary sign in Lemma \ref{lem:second_homotopy_sign}.   I would also like to thank Paul Seidel for explaining to me that a quotient $\cA / \cB$ of differential graded categories vanishes if and only if every object of $\cA$ is equivalent to a summand of a cone on objects of $\cB$.  The original proof of Theorem \ref{thm:main_thm} used a circuitous argument that relied on this fact.  Finally, I would like to thank the referee whose suggested changes have helped improve the exposition.
\subsubsection*{Conventions}
All chain complexes are complexes of free abelian groups, with the tensor product taken over $\bZ$ unless otherwise specified.  Throughout the paper, we shall be proving that certain expressions vanish by showing that their terms correspond to the boundary of $1$-dimensional moduli spaces.  Away from characteristic $2$, one must in addition verify that the natural orientations on the boundaries of these moduli spaces differ from certain product orientations by signs that are fixed by conventions for categories of $A_{\infty}$ bimodules.   We provide the ingredients which define the appropriate orientations, and verify the validity of the signs in the simplest situations.
\section{Geometric preliminaries} \label{sec:geom-prel}
Recall that a Liouville manifold is a manifold $M$ equipped with a $1$-form $\lambda$ whose differential $\omega$ is a symplectic form and such that the vector field $Z_{\lambda}$ defined by the equation
\begin{equation*}  i_{Z_{\lambda}} \omega = \lambda \end{equation*}
generates a complete expanding flow called the Liouville flow. Away from a compact set in $M$ we require that this flow be modeled after multiplication on the positive end of the symplectisation of a contact manifold.  We shall denote by $\psi^{\rho}$ the Liouville flow for time $\log(\rho)$.  In practice, we abuse notation and write
\begin{equation*}
\label{eq:complete_M}
M = M^{in} \cup_{\partial M}  [1, +\infty) \times \partial M,
\end{equation*}
where $\partial M$ is shorthand for $\partial M^{in}$.   Note that the Liouville form on the collar is given by $\lambda = r (\lambda| \partial M) $ with $ \lambda| \partial M $  a contact form and $r$ the coordinate on $[1,+\infty)$.  Our first transversality assumption will be imposed on the Reeb orbits of $ \lambda| \partial M $:
\begin{equation}
  \label{eq:non-degenerate_Reeb_orbit}
  \parbox{32em}{All Reeb orbits of $ \lambda| \partial M $ are non-degenerate.}
\end{equation}

\begin{lem}
Condition \eqref{eq:non-degenerate_Reeb_orbit} holds for a generic choice of $\lambda$.  For such a contact form, there are only finitely many Reeb chords shorter than any given constant.
\qed
\end{lem}

Next,  we consider a finite collection $\Ob(\Wrap)$  of exact connected properly embedded Lagrangians in $M$, such that 
\begin{equation} \label{eq:boundary_Legendrian} \parbox{32em}{$\lambda$ vanishes on $L \cap \partial M \times [1,+\infty)$ if $L \in \Ob(\Wrap)$ } \end{equation}
Note that this is precisely the condition that the intersection $\partial L$  of $L$ with $\partial M$ is Legendrian, and that $L$ is obtained by attaching an infinite cylindrical end to its intersection with $M^{in}$
\begin{equation*}  L  = L^{in} \cup_{\partial L}  [1, +\infty) \times \partial L. \end{equation*} 
We shall choose a primitive 
\begin{equation*}
f_L \co L \to \bR
\end{equation*}
for the restriction of $\lambda$ to each Lagrangian $L$ which, by Condition \eqref{eq:boundary_Legendrian} is necessarily locally constant on the cylindrical end of $L$.

Should we desire an integral grading on symplectic cohomology and the wrapped Fukaya category, we must impose the following additional condition on each object of $\Wrap$:
\begin{equation}
\label{eq:calabi-yauness}
\parbox{32em}{The relative first Chern class $2c_1(M,L) \in H^2(M,L)$  vanishes.}
\end{equation}
The first Chern class is taken for an almost complex structure on $M$ compatible with $\omega$; the notion is well-defined because the space of such almost complex structures is contractible.  Gradings in symplectic cohomology and in the wrapped Fukaya category require choices of
\begin{equation}
\label{eq:choice_volume_form_grading}
     \parbox{32em}{a trivialisation of $\left( \Lambda^{n}_{\bC} T^*M \right)^{\otimes 2}$, and a grading on each Lagrangian in $\Ob(\Wrap)$.}
\end{equation}
The vanishing of the first Chern class in Condition \eqref{eq:calabi-yauness} is precisely the obstruction to being able to make such a choice  (see \cite{seidel-GL} for details).  Note that a trivialisation  of $\Lambda^{n}_{\bC} T^*M$ is the same as the choice of a complex volume form on $M$ with respect to a compatible almost complex structure, so that we are requiring the choice of a quadratic volume form.

If, in addition,  we  desire for the Fukaya category to be defined over $\bZ$ (rather than only in characteristic $2$), we must assume that
\begin{equation}
\label{eq:spinness}
     \parbox{32em}{each Lagrangian $L \in \Ob(\Wrap)$ is spin.}
\end{equation}
Assuming this condition holds, we fix a spin structure (and orientation) on each Lagrangian.  For each object of $\Wrap$, a choice of grading and spin structures will be fixed from now on and go unmentioned.   

\subsection{Hamiltonians and almost complex structures}
We will work with a restricted family of functions $\sH(M) \subset C^{\infty}(M,\bR) $ such that, whenever $H \in \sH(M)$ we have
\begin{equation}
  \label{eq:quadratic_growth}
  H(r,y) = r^2
\end{equation}
away from some compact subset of $M$.

We fix such a function $H$, and write $X$ for its Hamiltonian flow and $\Chord(L_{0},L_1)$ for the set of time-$1$ flow lines of $X$ which start on a Lagrangian $L_0 \in \Ob(\Wrap)$ and end on $L_1  \in \Ob(\Wrap) $.  The relative analogue of \eqref{eq:non-degenerate_Reeb_orbit}, is the assumption that
\begin{equation}
  \label{eq:non-degenerate_chord}
    \parbox{32em}{all  time-$1$ Hamiltonian chords of $H$ with boundaries on $L_0$ and $L_1$  are non-degenerate.}
\end{equation}

\begin{lem} \label{lem:generic_Hamilt_chord}
For a generic Liouville form $\lambda$, Condition \eqref{eq:non-degenerate_chord} holds after a Hamiltonian perturbation of $L_0$ and $L_1$ which preserves Condition \eqref{eq:boundary_Legendrian}. \qed
\end{lem}
With this non-degeneracy assumption, the Maslov index defines an integral grading on the elements of $\Chord(L_0, L_1)$.  We shall write $\deg$ for the Maslov index.

Next, we consider the space $\sJ(M)$ of almost complex structures which are compatible with $\omega$, and whose restriction to the collar is of contact type in the sense that
\begin{equation}
  \label{eq:contact_type_complex}
  \lambda \circ J = dr.
\end{equation}

Given a family  $I_t \in  \sJ(M)$, parametrised by $t \in [0,1]$, we consider maps
\begin{equation*} u \co (-\infty, +\infty) \times [0,1] \equiv Z \to M  \end{equation*}
converging exponentially at each end to  time-$1$ periodic chords of $H$, with boundary conditions
\begin{align*}
u(s,0) \in & L_0 \\
u(s,1) \in & L_1
\end{align*}
and satisfying  Floer's equation
\begin{equation}
  \label{eq:dbar_strip}
  \left(du - X \otimes dt  \right)^{0,1} = 0.
\end{equation}
Here the strip is given  coordinates $(s,t)$ and complex structure  $j(\partial_s) = \partial_t$, so that  we may write Floer's equation as
\begin{equation*}
  \partial_{s}u = - I_{t} \left(\partial_tu - X \right).
\end{equation*}

Given a pair $x_0, x_1 \in  \Chord(L_0,L_1)$, we write $\Disct(x_0;x_1)$ for the set of such maps $u$ which converge to $x_0$ when $s$ approaches $-\infty$ and to $x_1$ at $+\infty$.   As this is a component of  the zero-locus of an elliptic operator on the space of smooth maps on the strip, it carries a natural topology.  Moreover,  since Equation \eqref{eq:dbar_strip} is invariant under translation in the $s$-variable, we obtain a continuous $\bR$ action on $\Disct(x_0;x_1)$.  Following \cite{FHS}, the usual transversality argument implies:
\begin{lem} \label{lem:free_R_action_strips}
For a generic family $I_t$, the moduli space  $\Disct(x_0;x_1)$ is regular; in particular it is a smooth manifold of dimension $\deg(x_0) - \deg(x_1)$.  Unless $\deg(x_0) = \deg(x_1)$, the action of $\bR$ is smooth and free. \qed
\end{lem}
We write $\Disc(x_0;x_1)$ for the quotient of $\Disct(x_0;x_1)$ by the $\bR$ action whenever it is free, and declare it to be the empty set otherwise.

Following Gromov and Floer, one may construct a bordification $\Discbar(x_0; x_1)$ by adding \emph{broken strips}
\begin{equation}
  \label{eq:compactification_strip}
  \Discbar(x_0; x_1)  = \coprod   \Disc(x_0; y_1) \times   \Disc(y_1; y_2)  \times \cdots \times  \Disc(y_k; y_{k+1})  \times  \Disc(y_{k+1}; x_1) 
\end{equation}

\begin{lem} \label{lem:compactification_strip_manifold}
For a generic family $I_t$, the moduli space  $\Discbar(x_0;x_1)$ is a compact manifold with boundary of dimension $\deg(x_0) - \deg(x_1) -1$.  Moreover, the boundary is covered by the closure of the images of the natural inclusions
\begin{equation*}   \Disc(x_0; y) \times   \Disc(y; x_1) \to     \Discbar(x_0; x_1).  \end{equation*} 
\end{lem}
\begin{proof}
The fact that $ \Discbar(x_0;x_1) $ is a manifold with boundary is a standard fact that does not have a clean proof in the literature.  In the case of $0$ and $1$ dimensional moduli spaces (which is the only part we need), the result essentially goes back to Floer (see Proposition 4.1 in \cite{floer}) and has been reproved in various settings since.  In order to prove the Lemma, is suffices therefore to prevent  solutions of the Cauchy-Riemann equation \eqref{eq:dbar_strip} from escaping to infinity.  However, the maximum principle prevents this eventuality from occurring.  Alternatively, we may apply Lemma \ref{lem:stay_in_compact_set}.
\end{proof}

Since there is no reason for $\Chord(L_0,L_1)$ to be finite, we shall require a stronger version of compactness. 
\begin{lem} \label{lem:finite_number_outputs_lag}
For each chord $x_1$, the moduli space $\Discbar(x_0;x_1)$ is empty for all but finitely many choices of $x_0$. 
\end{lem}
\begin{proof}
Choose a positive constant $R$ such that $\{R \} \times \partial M$ separates $x_1$ from infinity.  Since $H$ depends only on the radial variable along the neck, any element  $x \in \Chord(L_0,L_1)$ which intersects $ \{ R \}  \times  \partial M $ is contained in such a hypersurface, so Assumption \eqref{eq:non-degenerate_Reeb_orbit} implies that we may choose $R$ so that no chord intersects it.   As there are only finitely many chords in any compact subset of $M$, it suffices to prove that   $\Discbar(x_0;x_1)$ is empty whenever $x_0$ lies in $ [R , +\infty) \times \partial M $.  Using Lemma \ref{lem:action_negative}, this is an immediate consequence of Lemma \ref{lem:stay_in_compact_set} applied to $S = u^{-1}(  [R , +\infty) \times \partial M) $ and $v = u|S$.
\end{proof}

\section{Product and coproduct in wrapped Floer cohomology}
The construction of the wrapped Fukaya category described in \cite{abouzaid-seidel} is cumbersome for the purpose of defining the operation $\Delta$ introduced in Section \ref{sec:copr-wrapp-categ}.  Instead, we describe a different construction of the wrapped Fukaya category outlined in Section 3.2 of \cite{FSS}.  Our conventions are normalised so that (1) the differential raises degree and (2) the natural pair of pants product has degree $0$.

\subsection{Wrapped Floer cohomology}
For each pair $(L_0,L_1)$ of Lagrangians, we define the wrapped Floer complex to have underlying graded components
\begin{equation}
  \label{eq:wrapped_complex}
 CW^{i}(L_0,L_1; H, I_{t}) \equiv \bigoplus_{\stackrel{x \in \Chord(L_i,L_j)}{ \deg(x) = i}} | o_x|
\end{equation}
Here, $| o_x|$ is the orientation line of a certain rank-$1$ $\bR$-vector space $o_x$ associated to each chord (see Section (12b) of \cite{seidel-book}).  We shall omit $H$ and $I_{t}$ from the notation unless necessary.

In order to define a differential on the wrapped complex, we note that whenever $\deg(x_0) = \deg(x_1) +1 $, we obtain an isomorphism
\begin{equation*}  o_{x_1} \to o_{x_0}   \end{equation*}
from Equation \eqref{eq:orientation_differential_HW} and the trivialisation of the $\bR$ action on the moduli space $\Disc(x_0;x_1)$ explained in Remark \ref{rem:trivialise_R_action}.

Writing $\mu_u$ for the induced morphism on orientation lines, we define
\begin{align}
  \label{eq:wrapped_differential}
  \mu^{1} \co  CW^{i}(L_0,L_1)  & \to CW^{i+1}(L_0,L_1)  \\
[x_1] & \mapsto (-1)^i \sum_{u} \mu_u([x_1] )
\end{align}

Note that the right hand side is in fact a finite sum  because of Lemma \ref{lem:finite_number_outputs_lag}.  The proof that $\mu^1$ defines a differential is by now standard (see, e.g \cite{abouzaid-seidel}) and is omitted.  We call the resulting cohomology group \emph{wrapped Floer cohomology} and denote it by $HW^*(L_0, L_1)$.

\subsection{Composition in the wrapped category}\label{sec:comp-wrapp-categ}
Note that pullback of the moduli space of solutions to Equation \eqref{eq:dbar_strip} by the Liouville flow for time $\log(\rho)$ defines a canonical isomorphism
\begin{equation}
  \label{eq:isomorphic_complexes_rescale}
CW^{*}(L_0,L_1; H, I_{t}) \cong   CW^{*}\left( \psi^{\rho} L_0, \psi^{\rho} L_1;   \frac{H}{\rho} \circ  \psi^{\rho} ,  (\psi^{\rho})^{*}I_{t} \right),
\end{equation}
where the right hand side is computed with respect to the symplectic form $\omega$ (rather than its pullback under $ \psi^{\rho}$), so that the Hamiltonian flow of $  \frac{H}{\rho} \circ  \psi^{\rho} $ is precisely the pullback of $ X_H$.  The main observation required to define an $A_{\infty}$ structure on the complexes $CW^{*}(L_0,L_1)$ is:
\begin{lem} \label{lem:rescaled_function_admissible}
The function  $  \frac{H}{\rho^2} \circ  \psi^{\rho} $ lies in $\cH(M)$.
\end{lem}
\begin{proof}
This is an elementary computation using the fact that the Liouville flow is given on the collar by
\begin{equation}
 \label{eq:Liouville_flow_collar}
 \psi^{\rho} (r,y) =  (\rho \cdot r,y) .
\end{equation}
In particular, $  r^{2} \circ  \psi^{\rho} = \rho^{2} r^{2} $.  
\end{proof}

Given a triple of Lagrangians $L_0$, $L_1$ and $L_2$, we shall define a chain map
\begin{equation*}
CW^{*}(L_1, L_2; H, I_{t})  \otimes CW^{*}(L_0,L_1; H, I_{t})  \to   CW^{*}\left( \psi^{2} L_0, \psi^{2} L_2;   \frac{H}{2} \circ  \psi^{2} ,  (\psi^{2})^{*}I_{t} \right),
\end{equation*}
which when composed with the inverse of the isomorphism of Equation \eqref{eq:isomorphic_complexes_rescale} gives the product
\begin{equation}
 \label{eq:cup_product}
\mu^2 \co CW^{*}(L_1, L_2)  \otimes CW^{*}(L_0,L_1)  \to   CW^{*}( L_0,  L_2)
\end{equation}
in the wrapped Fukaya category.

Let $S$ denote the complement of three points $(\xi^0,\xi^1,\xi^2)$  ordered counter-clockwise on the boundary of $D^2$.  The operation $\mu^{2}$  will be defined by counting solutions to an elliptic equation on $S$ analogous to Equation \eqref{eq:dbar_strip} for which we must make the following choices:  Fix a map $\rho_{S}$ from the boundary of $D^2$ to the interval  $[1,2]$   such that
\begin{equation}
  \label{eq:condition_arc_labelling}
  \parbox{32em}{$\rho_S(z) = 1$ if $z$ is near $\xi^{1}$ or  $\xi^{2}$ and $\rho_S(z) = 2$ if $z$ is near $\xi^0$.} 
\end{equation}

Let $Z_+$ and $Z_-$ denote the positive and negative half-strips  in $Z$ and choose embeddings 
\begin{align*}
\epsilon^{0} \co Z_- & \to S \\
 \epsilon^k \co Z_+ & \to S \textrm{ if k=1,2} 
\end{align*}
which map $\partial Z_{\pm}$ to $\partial S$ and converge to the respective marked points $\xi^k$.  This will be called a choice of \emph{strip-like ends} for $S$.

Choose a \emph{closed} $1$-form $\alpha_{S}$ on $S$ whose restriction to the boundary vanishes, and whose pullback under $  \epsilon^k $ agrees with $dt$ if $k=1,2$ and with $2dt$ if $k=0$.  In addition, choose maps
\begin{align*}
 H_{S} & \co S  \to \sH(M) \\
  I_{S} & \co S  \to \sJ(M)
\end{align*}
whose compositions with $\epsilon^k$ agree with $H$ and $I_{t}$ if $k = 1,2$, and with $\frac{H \circ \psi^{2}}{4}$ and $(\psi^{2})^{*}I_{t}$ if $k=0$.  Given $z \in S$, we write $X_{S}(z)$ for the Hamiltonian flow of $H_{S}(z)$.

We write $\Disc_{2}(x_0;x_1,x_2)$ for the space of maps from $S$ to $M$ with the following boundary conditions
\begin{equation}
  \label{eq:boundary_disc_2_punctures}
  \begin{cases}  u(z)  \in \psi^{\rho_{S}(z)} L_1 & \textrm{if $z \in \partial S$ lies between $\xi_1$ and $\xi_2$} \\
  u(z)  \in \psi^{\rho_{S}(z)} L_2 & \textrm{if $z \in \partial S$ lies between $\xi_2$ and $\xi_0$} \\
  u(z)  \in \psi^{\rho_{S}(z)} L_0 & \textrm{if $z \in \partial S$ lies between $\xi_1$ and $\xi_0$} \\
\lim_{s \to \pm \infty} u \circ \epsilon^{k}(s, \cdot) = \psi^{\rho_{S}(\xi^k)}x_k(\cdot) &  \textrm{for $k=0,1,2$.}
 \end{cases}
\end{equation}
and solving the differential equation
\begin{equation}
  \label{eq:dbar_pair_pants}
  \left( du - X_S \otimes \alpha_S \right)^{0,1} = 0.
\end{equation}
In words, the last condition in \eqref{eq:boundary_disc_2_punctures} says that the image of $u$ converges to $x_{1}$ and $x_2$ at the punctures corresponding to inputs, and to $  \psi^{2}x_0$  at the one corresponding to the output.

\begin{lem}
For a generic family of almost complex structures $I_S$, the moduli space $\Disc_{2}(x_0;x_1,x_2)$ is a smooth manifold of dimension
\begin{equation*}
  \deg(x_0) - \deg(x_1) - \deg(x_2)
\end{equation*}
which, for a fixed pair $(x_1,x_2)$ is empty for all but finitely many choices of chords $x_0$.    Its Gromov bordification $\Discbar_{2}(x_0;x_1,x_2)$  is a compact manifold whose boundary is covered by the codimension $1$ strata
\begin{multline}
  \label{eq:cover_boundary_disc_2}
 \coprod_{y \in \Chord(L_0,L_1)}  \Disc_{2}(x_0;y,x_2) \times \Disc(y;x_1)  \cup   \coprod_{y \in \Chord(L_1,L_2)}  \Disc_{2}(x_0;x_1,y) \times \Disc(y;x_2) \\  \cup   \coprod_{y \in \Chord(L_0,L_2)}  \Disc(x_0;y) \times \Disc_{2}(y;x_1,x_2) .
\end{multline}
\end{lem}
\begin{proof}
We omit the proof of transversality, which follows from a standard Sard-Smale argument.  Since the pullback of Equation \eqref{eq:dbar_pair_pants} under $\epsilon^k$ is Equation \eqref{eq:dbar_strip} for Floer data $(H,I_t)$ whenever $k = 1,2$, and Floer data $(  \frac{H}{2} \circ  \psi^{2} ,  (\psi^{2})^{*}I_t )$ whenever $k=0$, we conclude  that  $\Discbar_{2}(x_0;x_1,x_2)$ is indeed obtained by adding the strata in Equation \eqref{eq:cover_boundary_disc_2} to the moduli space of solutions of Equation \eqref{eq:dbar_pair_pants}.

The fact $\Discbar_{2}(x_0;x_1,x_2) $ is empty for all but finitely many choices of $x_0$ follows from Lemma \ref{lem:stay_in_compact_set} applied to $S= u^{-1}([R,+\infty) \times \partial M)$ where $R$ is chosen so that $\{ R \} \times \partial M$ separates the inputs form infinity (this is the same argument which proves Lemma \ref{lem:finite_number_outputs_lag}).    This argument also shows that  all elements of $\Disc_{2}(x_0;x_1,x_2)$  have images contained in a compact subset of $M$, which allows us to apply the usual proof of Gromov compactness.
\end{proof}

We conclude that whenever $  \deg(x_0)  = \deg(x_1) + \deg(x_2) $, the elements of $\Disc_{2}(x_0;x_1,x_2)$ are rigid.  By Equation \eqref{eq:orientation_higher_product}, every such element induces an isomorphism
\begin{equation*}
  o_{x_2} \otimes o_{x_1} \to o_{x_0},
\end{equation*}
and hence a map $\mu_u$  on orientation lines.   The cup product in the wrapped Fukaya category is obtained by taking the sum of all such maps
 \begin{align*}
\mu^2 \co CW^{*}(L_1, L_2)  \otimes CW^{*}(L_0,L_1 ) & \to   CW^{*}( L_0, L_2)  \\
\mu^2([x_2],[x_1]) & = \sum_{ \stackrel{ \deg(x_0)  = \deg(x_1) + \deg(x_2) }{ u \in \Disc_{2}(x_0;x_1,x_2)} } (-1)^{\deg(x_1)}  \mu_u([x_2],[x_1])
\end{align*}
Since the signs are exactly the same as in \cite{seidel-book} and \cite{abouzaid-seidel}, we omit the proof of their validity.

\subsection{A coproduct on the wrapped category} \label{sec:copr-wrapp-categ}
In this section, we define a degree $n$ chain map
\begin{equation} \label{eq:first_term_coproduct}
  \Delta^{0|1|0} \co  CW^{*}(L_{|}, L)  \to  \bigoplus_{p+q=* +n}  CW^{p}(K, L )  \otimes   CW^{q}(L_{|}, K) ,
\end{equation}
where the differential on the left hand side is $- \mu^{1}$, while the one on the right hand side is given by
\begin{equation}
  \label{eq:differential_tensor_over_k}
  a_{-1} \otimes a_0 \mapsto (-1)^{\deg(a_0)+1} \mu^{1}(a_{-1}) \otimes a_0  - a_{-1} \otimes  \mu^{1}(a_{0}).
\end{equation}

We write $T$ for the punctured surface obtained by removing $3$ points from the boundary of $D^2$; the three marked points, ordered counter clockwise will now be denoted $(\xi^{-1}, \xi^{0}, \underline{\xi})$.   If we were only giving a sketch of the construction, we would say that $\Delta^{0|1|0}$ counts  holomorphic maps from $T$ to $M$ with the boundary condition as shown at the bottom right of Figure \ref{fig:Cardy}.

In order to write the precise equation whose solutions we count, fix a map $\rho_{T}$ from the boundary of $D^2$ to the interval  $[1,1/2]$   such that
\begin{equation}
  \label{eq:condition_arc_labelling_coproduct}
  \parbox{32em}{$\rho_T(z) = 1$ if $z$ is near $\underline{\xi}$ and $\rho_T(z) = 1/2$ if $z$ is near $\xi_k$.} 
\end{equation}
Choose a positive strip-like end $\underline{\epsilon}$ at $\underline{\xi}$, and negative strip-like ends  $  \epsilon^k $ at the other marked points, as well as a closed $1$-form $\alpha_{T}$  whose restriction to the boundary vanishes, and whose pullback under $\underline{\epsilon}$ agrees with $dt$ while the pullbacks under the other ends agree with $\frac{dt}{2}$.   Finally, choose maps
\begin{align*}
  I_{T} & \co T \to \sJ(M)  \\
 H_{T} & \co T \to \sH(M)
\end{align*}
whose compositions with $\underline{\epsilon}$ agree with $I_{t}$ and $H$, while the compositions with $ \epsilon^{k}$ agree with  $  (\psi^{1/2}) ^{*}  I_{t} $ and  $ 4 (\psi^{1/2}) ^{*}H $.  We write $X_{T}$ for the Hamiltonian flow of $H_T$.

Given time-$1$ chords  $ x_{-1} \in \Chord(L_{|},K)$,  $x_0 \in \Chord(K,L)$ and $\underline{x} \in \Chord(L_{|},L)$,  we write $\Disc_{0|1|0}(x_{-1},x_0;\underline{x})$ for the space of maps from $T$ to $M$ with the following boundary and asymptotic conditions (keep in mind the idealisation of the boundary conditions, if $L=L_{|}$,  shown in Figure \ref{fig:Cardy} )
\begin{equation}
  \label{eq:boundary_disc_2_outputs}
  \begin{cases}  u(z)  \in \psi^{\rho_{S}(z)}  K & \textrm{if $z \in \partial T$ lies between $\xi^{-1}$ and $\xi^0$} \\
  u(z)  \in   \psi^{\rho_{S}(z)} L_{|} & \textrm{if $z \in \partial T$ lies between $\xi^0$ and $\underline{\xi}$} \\
  u(z)  \in  \psi^{\rho_{S}(z)}  L & \textrm{if $z \in \partial T$ lies between $\underline{\xi}$ and $\xi^{-1}$} \\
\lim_{s \to \pm \infty} u \circ \epsilon^{k}(s, \cdot) =  \psi^{\rho_{S}(z)} x_k(\cdot) &  \textrm{for $k=-1,0$} \\
\lim_{s \to \pm \infty} u \circ \underline{\epsilon}(s, \cdot) = \psi^{\rho_{S}(z)}  \underline{x}(\cdot) &
 \end{cases}
\end{equation}
and solving the differential equation
\begin{equation}
  \label{eq:dbar_pair_pants_coprod}
  \left( du - X_{T} \otimes \alpha_T  \right)^{0,1} = 0.
\end{equation}

The standard transversality and compactness theorems hold with the same proofs as for the moduli space $\Disc_{2}(x_0;x_1,x_2)$ so that whenever $  \deg(\underline{x})  +n = \deg(x_{-1}) + \deg(x_0) $ this moduli space is rigid.    Fixing an orientation for $K$, Equation \eqref{eq:orientation_coproduct} yields an isomorphism
\begin{equation*}
  o_{\underline{x}} \to   o_{x_{-1}}  \otimes  o_{x_0}
\end{equation*}
associated to every map $u \in \Disc_{0|1|0}(x_{-1},x_0;\underline{x})$.  We write $\Delta^{0|1|0}_u$  for the induced map on orientation lines. The energy estimate of Lemma \ref{lem:stay_in_compact_set} implies that for a fixed input $\underline{x}$, there are only finitely many possible outputs.  We obtain Equation \eqref{eq:first_term_coproduct} as a sum
\begin{equation}
  \label{eq:coproduct_first_term}
 \Delta^{0|1|0} =  \sum  \Delta^{0|1|0}_u
\end{equation}
 of contributions over all such discs in rigid moduli spaces $\Disc_{0|1|0}(x_{-1},x_0;\underline{x})$.  Note that there are no signs in this expression.  
\begin{lem} \label{lem:Delta_chain_map}
$ \Delta^{0|1|0}  $ is a chain map of degree $n$.
\end{lem}
\begin{proof}
The boundary of  $\Disc_{0|1|0}(x_{-1},x_0;\underline{x})$ is covered by the images of the products 
 \begin{align*}
&  \Disc_{0|1|0}(x_{-1},x_0;\underline{y}) \times \Disc(\underline{y}; \underline{x}) \\
&  \Disc(x_0; y_0) \times \Disc_{0|1|0}(x_{-1},y_0;\underline{x})  \\
& \Disc(x_{-1}; y_{-1}) \times \Disc_{0|1|0}(y_{-1},x_0;\underline{x}).
 \end{align*}
According to Equations \eqref{eq:orientation_differential_HW} and \eqref{eq:orientation_coproduct}, the orientations on the factors of the first stratum induce natural isomorphisms
\begin{align*}
o_{\underline{y}} &   \cong  \lambda(\Disc_{0|1|0}) \otimes  o_{x_{-1}} \otimes  o_{x_{0}} \otimes  \lambda^{-1}(L_0)  \\
o_{\underline{x}} &   \cong   \lambda^{-1}( \Disct(\underline{y};\underline{x}) )   \otimes  o_{\underline{y}} 
\end{align*}
Composing these two isomorphisms, we obtain 
\begin{equation*}
o_{\underline{x}}   \cong   \lambda^{-1}( \Disct(\underline{y};\underline{x}) )   \otimes   \lambda(\Disc_{0|1|0}) \otimes  o_{x_{-1}} \otimes  o_{x_{0}} \otimes  \lambda^{-1}(L_0)
\end{equation*}
Note that for the purpose of defining the differential, $  \lambda^{-1}( \Disct(\underline{y};\underline{x}) )   $ is trivialised using the vector $\partial_{s}$, which after gluing gives the inward pointing vector to $\Disc_{0|1|0}(x_{-1},x_0;\underline{x})  $.  We conclude that the map $\Delta^{0|1|0} \circ (-\mu^{1})$  differs from the map induced by natural boundary orientation by a sign whose parity is
\begin{equation} \label{eq:sign_Delta_chain_map-1}
  \deg(\underline{x}).
\end{equation}
Next, we perform the same analysis on the second type of boundary stratum, and find that the product orientation is given by
\begin{equation*}
 o_{\underline{x}} \cong  \lambda(\Disc_{0|1|0}) \otimes  o_{x_{-1}}   \otimes    \lambda^{-1}( \Disct(x_0;y_0) )   \otimes  o_{x_0} \otimes  \lambda^{-1}(L_0)
\end{equation*}
Rearranging the terms introduces a sign of $\deg(x_{-1})$, and the gluing parameter $\partial_{s}$ now agrees with the outward pointing vector.  We conclude that the map $- (\id  \otimes \mu^1 )  \circ \Delta^{0|1|0} $ differs from the map induced by the boundary by 
\begin{equation}\label{eq:sign_Delta_chain_map-2}
 \deg(x_{-1}) + \deg(x_{0}) + 1 =  \deg(\underline{x}) +n +1.
\end{equation}
Finally, on the third stratum, the product orientation gives an isomorphism

\begin{equation*}
o_{\underline{x}} \cong  \lambda(\Disc_{0|1|0}) \otimes    \lambda^{-1}( \Disct(x_0;y_0) )  \otimes  o_{x_{-1}}     \otimes  o_{x_0}  \otimes  \lambda^{-1}(L_0)
\end{equation*}
which exactly agrees with the map induced by the boundary orientation, so that $(-1)^{\deg(x_0)+1}(\mu^1 \otimes \id   )  \circ \Delta^{0|1|0} $ differs from the map induced by the boundary by 
\begin{equation*}
 \deg(x_{-1}) + \deg(x_{0}) + 1 =  \deg(\underline{x}) +n +1.
\end{equation*}
Comparing this expression with Equations \eqref{eq:sign_Delta_chain_map-1} and \eqref{eq:sign_Delta_chain_map-2}, we readily conclude the desired result.
\end{proof}
\section{$\infty$-refinements of the product and coproduct}
\subsection{$A_{\infty}$ structure} \label{sec:a_infty-structure}

In this section, we construct the higher products
\begin{equation*}
  \mu^d \co  CW^{*}(L_{d-1}, L_d)    \otimes \cdots \otimes    CW^{*}(L_1, L_2)  \otimes  CW^{*}(L_0,L_1 )  \to   CW^{*}( L_0, L_d)
\end{equation*}
defining the $A_{\infty}$ structure on the wrapped Fukaya category by studying parametrised moduli spaces of solutions to a family of equations analogous to Equation \eqref{eq:dbar_pair_pants}.

\begin{defin} \label{def:floer_datum_disc_1_output}
A \emph{Floer datum} $D_{S}$ on a stable disc $S \in \Discbar_{d}$ with one negative and  $d$ positive ends consists of the following choices on each component:
\begin{enumerate}
\item Strip-like ends near the marked points:  We have a parametrisation  $\epsilon^k \co Z_{\pm} \to S$  of each end $\xi^k \in \overline{S}$  whose source is $Z_-$ if $k=0$ and $Z_+$ otherwise.
\item Time shifting map:  A map $ \rho_{S} \co \partial \bar{S} \to [1,+ \infty)$ which is constant near each end.  We write $w_{k,S}$ for the value on the $k$\th end. 
\item Basic $1$-form and Hamiltonian perturbations:  A closed $1$-form $\alpha_{S}$ whose restriction to the boundary vanishes and a map $H_{S} \co S \to \sH(M)$  on each surface defining a Hamiltonian flow $X_{S}$  such that the pullback of $ X_{S} \otimes \alpha_{S} $ under $ \epsilon^k $ agrees with  $  X_{\frac{H}{w_{k,S}} \circ \psi^{w_{k,S}}} \otimes dt$.
\item Almost complex structures:  A map $I_{S} \co S \to \sJ(M)$ whose pullback under $ \epsilon^k $ agrees with  $ (\psi^{w_{k,S}})^{*} I _t$.
\end{enumerate}
\end{defin}
The condition on the basic $1$-forms and Hamiltonian perturbations can be split into two conditions: The $1$-forms should pullback to $ w_{k,S} dt$  under $ \epsilon^k$ which in particular implies that
\begin{equation*} w_{0,S} = \sum_{1 \leq k \leq d} w_{k,S} \end{equation*}
while the function $H_{S}$ should pullback to $ \frac{H \circ \psi^{w_{k,S}}}{w^{2}_{k,S}}$  which is indeed an element of $ \sH(M) $  by Lemma \ref{lem:rescaled_function_admissible}. 

Points in a codimension $k$ stratum $\sigma \subset \partial  \Discbar_{d}$ represent stable curves with $k$ nodes.   Let us choose strip-like ends on curves in lower dimensional moduli spaces varying smoothly with respect to the modulus, and which are compatible with gluings of Riemann surfaces for sufficiently large gluing parameters.   In particular, for each positive real number $R$, we obtain an element of $  \Discbar_{d} $  by removing the images of $(-\infty, -R] \times [0,1]$ and $[R, +\infty)  \times [0,1]$  for the two sides of the node, and gluing the complements to form a new Riemann surface.    By applying this construction at every node, we obtain a  chart on $\Discbar_{d} $
\begin{equation} \label{eq:corner_chart_moduli_space}
  (0,+\infty]^{k} \times \sigma \to \Discbar_{d}. 
\end{equation}
whose image is an open neighbourhood of $\sigma$, such that each coordinate corresponds to a gluing parameter.

We shall consider a notion of equivalence among Floer data which is weaker than equality.  We say that $D^1_S$ and $D^2_S $  are \emph{conformally equivalent} if there exists a constant $C$ so that $\rho^{2}_{S}$ and $\alpha_{S}^{2}$ respectively agree with   $C \rho^{1}_{S}$ and $C \alpha_{S}^{1}$, and
\begin{align*}
 I_{S}^{2} & = {\psi^{C}}^{*} I_{S}^{1} \\
H_{S}^{2} & = \frac{H_{S}^{1} \circ \psi^{C}}{C^{2} }.
 \end{align*}

\begin{defin}
A   \emph{universal and conformally consistent}  choice of Floer data for the $A_{\infty}$ structure  is a choice $\bfD_{\mu}$  of Floer data  for every integer $d \geq 2$ and every (representative) of an element of  $\Discbar_{d}$ which varies smoothly over this compactified moduli space, and whose restriction to a boundary stratum is conformally equivalent to the product of Floer data coming from lower dimensional moduli spaces.  Moreover,  in the coordinates \eqref{eq:corner_chart_moduli_space}, Floer data agree to infinite order at boundary stratum with the Floer data obtained by gluing.
\end{defin} 
This consistency condition implies that the choices  for $\partial \Discbar_{d}$  are determined by those made for lower dimensional moduli spaces up to a choice of  the conformal equivalence constant on each irreducible component.  
\begin{figure}[h]
   \centering
   \includegraphics{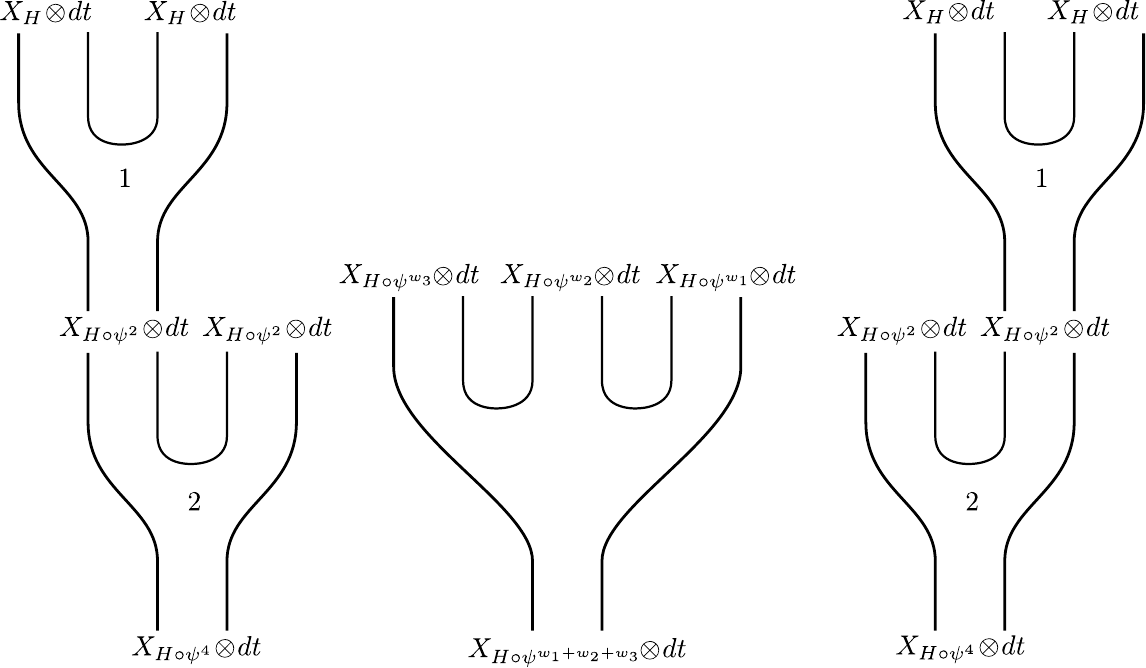}
\caption{}
\label{fig:floer_data_three_punctures}
 \end{figure}
In particular, a choice of Floer data for $d=2$ (already performed in Section \ref{sec:comp-wrapp-categ}) determines such a choice of Floer data for the two points which form the boundary of $ \Discbar_{3}$   up to some conformal scaling constant.  The simplest choice yields the scaling constant $2$ for the disc which includes the outgoing end and $1$ for the other disc,  as displayed in Figure \ref{fig:floer_data_three_punctures},  which shows the pullback of the product $X_{H_{S}} \otimes \alpha_{S}$  at various ends.  By gluing, we obtain Floer data for elements of $\Disc_{3}$ which are sufficiently close to the boundary.  We choose a perturbation of this glued data which vanishes to infinite order near the boundary, then, using the fact that the space of allowable Floer data is contractible, we choose Floer data on the remaining part of $\Disc_{3}$.  We proceed inductively using the covering of  $\partial \Discbar_{d}$ by images of codimension $1$ inclusions
\begin{equation}
\Discbar_{d_1} \times \Discbar_{d-d_1+1} \to  \partial \Discbar_{d}
\end{equation}
and conclude
\begin{lem} \label{lem:abundance_floer_data}
The restriction map from the space of universal and conformally consistent Floer data to the space of Floer data for a fixed surface $S$ is surjective.  \qed
\end{lem}

Let $L_0, \ldots,  L_{d}$ denote objects of $\Wrap$, and consider a sequence of chords $\vx = \{ x_{k} \in \Chord(L_{k-1},L_k) \} $ if $1 \leq k \leq d  $ and $x_0 \in \Chord(L_0,L_d)$.  Given universal Floer data $\bfD_{\mu}$, we write $\Disc_{d}(x_0; \vx) $ for the space of maps $u \co S \to  M$ whose source is an arbitrary element $S \in \Disc_{d}$ with marked points $(\xi^0, \ldots, \xi^{d}) $ and which satisfy the boundary and asymptotic conditions
\begin{equation} \label{eq:boundary_condition_disc_A_infty}
\begin{cases}
u(z) \in \psi^{\rho_{S}(z)} L_k & \textrm{if $z \in \partial S$ lies between $\xi^k$ and $\xi^{k+1}$} \\
\lim_{s \to \pm \infty} u \circ \epsilon^{k}(s, \cdot) =  \psi^{\rho_{S}(z)} x_k &
\end{cases}
\end{equation}
and the differential equation
\begin{equation}
  \label{eq:dbar_disc_1_output}
\left(du - X_{S} \otimes \alpha_{S}\right)^{0,1} = 0
\end{equation}
where the $(0,1)$ part is taken with respect to the $S$-dependent almost complex structure.

The consistency condition imposed on  $\bfD_{\mu}$ implies that the Gromov bordification $\Discbar_{d}(x_0; \vx) $  is obtained by adding the images of natural inclusions
\begin{equation}
  \label{eq:codim_1_strata_discs_1_output}
  \Discbar_{d_1}(x_0; \vx[1])  \times   \Discbar_{d_2}(y; \vx[2]) \to  \Discbar_{d}(x_0; \vx)
\end{equation}
where $y$ agrees with one of the elements of $\vx[1] $, and the sequence $\vx$ is obtained by removing $y$  from $\vx[1]$ and replacing it by the sequence $\vx[2]$.

Applying Lemma \ref{lem:stay_in_compact_set} to achieve compactness, the standard Sard-Smale argument to achieve transversality (see for example Section (9k) of \cite{seidel-book}), and the index theorem to compute expected dimensions, we conclude:
\begin{lem}
The moduli spaces $\Discbar_{d}(x_0; \vx) $  are compact.  In addition, for a generic choice  $\bfD_{\mu}$, they form manifolds of dimension
\begin{equation*}
  \deg(x_0) + d - 2 - \sum_{1 \leq k \leq d} \deg(x_k).  
\end{equation*}
\qed
\end{lem}

Assuming that $\deg(x_0) =  2 - d + \sum_{1 \leq k \leq d} \deg(x^k)  $, we conclude that the elements of  $\Disc_{d}(x_0; \vx) $ are rigid.  Choosing the orientation of $ \Disc_{d} $ fixed in Appendix \ref{sec:stash-polyh-infty}, we obtain a canonical homotopy class of isomorphisms 
\begin{equation}
  o_{x_d} \otimes \cdots \otimes o_{x_1} \to o_{x_0}
\end{equation}
from Equation \eqref{eq:orientation_higher_product}. Writing $\mu_{u}$ for the induced map on orientation lines, we define  an operation
\begin{equation*}
  \mu^d \co  CW^{*}(L_{d-1}, L_d)    \otimes \cdots \otimes    CW^{*}(L_1, L_2)  \otimes  CW^{*}(L_0,L_1 )  \to   CW^{*}( L_0, L_d)
\end{equation*}
called the $d$\th higher product as a sum
\begin{equation}
  \mu^{d}([x_d], \ldots, [x_1])  = \sum_{\stackrel{\deg(x_0) =  2 - d + \sum_{1 \leq k \leq d} \deg(x_k)}{u \in \Disc_{d}(x_0; \vx) }} (-1)^{\dagger} \mu_{u}( [x_d], \ldots, [x_1])
\end{equation}
where the sign is given by
\begin{equation} \label{eq:dagger_sign}
 \dagger = \sum_{k=1}^{d} k \deg(x_k).
\end{equation}

\begin{prop} \label{prop:a_infty_structure}
The sum of quadratic compositions of the higher products
\begin{equation}
  \label{eq:a_infty_property}
   \sum_{\stackrel{d_1 + d_2 = d +1}{0 \leq k <d_1}} (-1)^{\maltese_{1}^{k}} \mu^{d_1}\left(x_d, \ldots, x_{k+d_2+1}, \mu^{d_2}(x_{k+d_2}, \ldots, x_{k+1}) ,   x_k, \ldots , x_1 \right) = 0 
\end{equation}
vanishes, where the value of the sign is given by
\begin{equation*} \maltese_{1}^{k} = \sum_{1 \leq j \leq k} ||x_j||. \end{equation*}
and $||x_j|| = \deg(x_j) +1$ is the reduced degree.  In particular, the operations $\mu^d$ form an $A_{\infty}$ structure. \qed
\end{prop}

\subsection{The coproduct as a map of bimodules}\label{sec:coproduct-as-map}
Let $\cB$ be a full subcategory of $\Wrap$.   Recall that a bimodule $\cP$ over $\cB$ is a collection of graded abelian groups $\cP(L_{|},L)$ for all pairs of objects of $\cB$,  equipped with operations
\begin{multline*}
  \mu^{r|1|s} \co CW^{*}(L_{r-1}, L_r) \otimes \cdots \otimes CW^{*}(L_0, L_1) \otimes \\ \cP(L_{|0}, L_{0}) \otimes  CW^{*}( L_{|1}, L_{|0}) \otimes \cdots \otimes CW^{*}(L_{|s}, L_{|s-1}) \to  \cP(L_{|s}, L_r) 
\end{multline*}
satisfying the quadratic equation
\begin{multline*}
 \sum (-1)^{\maltese_{|s}^{|\ell+1}}   \mu^{r-m|1|s-\ell}(a_r, \ldots, a_{m+1}, \mu^{m|1|\ell}(a_m, \ldots, \underline{p}, \ldots, a_{|\ell}), a_{|\ell+1}, \ldots, a_{|s}  ) \\
+ \sum  (-1)^{\maltese_{|s}^{|\ell+1}}   \mu^{r|1|s-\ell + k +1 }(a_r, \ldots, \underline{p}, \ldots, a_{|k}, \mu^{\ell-k}(a_{|k+1}, \ldots, a_{|\ell}), a_{|\ell+1}, \ldots, a_{|s}  )  \\
+ \sum (-1)^{\maltese_{|s}^{k}}   \mu^{r-m +k  + 1 |1|s}(a_r, \ldots, a_{m+1}, \mu^{m-k}(a_m, \ldots, a_{k+1}), a_k,  \ldots,  \underline{p}, \ldots,  a_{|s}  ) = 0 
\end{multline*}
where the signs are given by
\begin{align*}
\maltese_{|s}^{|\ell+1} & = \sum_{j=\ell+1}^{s} || a_{|j} || \\
\maltese_{|s}^{k} & = \deg(\underline{p})  +  \sum_{1 \leq j \leq s} ||a_{|j}|| + \sum_{1 \leq j \leq k}||a_k||.
 \end{align*}
\begin{rem}
To make sense of the notation, consider the ordering $ |s < |s-1 < \ldots < |1 < \_  < 1 < \ldots < r$, and assign the elements of $CW^*(L_k, L_{k+1})$ or  $CW^*(L_{|k}, L_{|k-1})$ their reduced degree, while $\underline{p}$ is assigned its usual degree.  With this in mind, $\maltese_{*}^{**}$ is the sum of the degrees of elements between $*$ and $**$.  For more detail, the reader may consult \cite{seidel-bimodules} whose conventions we have adopted up to reversing the order of the inputs, and other minor changes of notation.
\end{rem}

 Via the Yoneda embedding, any object $K$ of $\Wrap$ defines a left and a right module over $\cB$ which we shall denote respectively $\cY^l_{K}$ and $ \cY^r_{K} $, and which associate to any object $L$  the graded abelian groups
\begin{align}
  \cY^l_{K} (L)  & = CW^{*}(K,L)  \\
  \cY^r_{K} (L)  & = CW^{*}(L,K).
\end{align}

In particular the tensor product  $ \cY^l_{K} \otimes \cY^{r}_{K}$ is an $A_{\infty}$-bimodule over $\cB$.  The differential is given by Equation \eqref{eq:differential_tensor_over_k}, while the higher operations vanish unless either $r$ or $s$ are equal to $0$, in which case we find that 
\begin{align*}
  \label{eq:bimodule_action_tensor}
\mu^{0|1|s}(p \otimes q, a_{|1}. \ldots, a_{|s}) &  = (-1)^{\maltese_{|s}^{|1} +1} p \otimes \mu^{s+1}(q, a_{|1}. \ldots, a_{|s})   \\
 \mu^{r|1|0}(a_r, \ldots, a_1, p \otimes q) &  =(-1)^{\deg(q) +1}   \mu^{r+1}(    a_r, \ldots, a_1, p) \otimes q.
\end{align*}

The main result we shall prove in this section  relates $ \cY^l_{K} \otimes_{\bfk} \cY^{r}_{K}$ to the diagonal bimodule $\cB$ with operations $\mu^{r|1|s} = (-1)^{\maltese_{|s}^{|1} +1} \mu^{r+s+1} $:
\begin{prop}  \label{prop:map_bimodules}
The  map $\Delta^{0|1|0}$ extends to a degree $n$ homomorphism of $A_{\infty}$ bimodules
\begin{equation}
  \label{eq:bimodule_map}
 \Delta \co  \cB \to  \cY^l_{K} \otimes \cY^{r}_{K}.
\end{equation}
\end{prop}
By definition, $\Delta$ consists of a collection of maps
\begin{multline*}
  \Delta^{r|1|s} \co CW^{*}(L_{r-1}, L_r) \otimes \cdots \otimes CW^{*}(L_0, L_1) \otimes CW^{*}(L_{|0}, L_{0}) \otimes  \\
CW^{*}( L_{|1}, L_{|0}) \otimes \cdots \otimes CW^{*}(L_{|s}, L_{|s-1}) \to    CW^{*}(K, L_{r}) \otimes  CW^{*}(L_{|s}, K )
\end{multline*}
satisfying the quadratic equation
\begin{multline} \label{eq:infinity_bimodule_equation}
 \sum (-1)^{n \maltese_{|s}^{|\ell+1}}   \mu^{r-m|1|s-\ell}(a_r, \ldots, a_{m+1}, \Delta^{m|1|\ell}(a_m, \ldots, \underline{p}, \ldots, a_{|\ell}), a_{|\ell+1}, \ldots, a_{|s}  ) \\
+  \sum (-1)^{  \maltese_{|s}^{|\ell+1} + n +1}   \Delta^{r-m|1|s-\ell}(a_r, \ldots, a_{m+1}, \mu^{m|1|\ell}(a_m, \ldots, \underline{p}, \ldots, a_{|\ell}), a_{|\ell+1}, \ldots, a_{|s}  ) \\
+ \sum (-1)^{  \maltese_{|s}^{|\ell+1} +n +1}   \Delta^{r|1|s-\ell + k }(a_r, \ldots, \underline{p}, \ldots, a_{|k}, \mu^{n-k}(a_{|k+1}, \ldots, a_{|\ell}), a_{|\ell+1}, \ldots, a_{|s}  )  \\
+ \sum (-1)^{\maltese_{|s}^{k} +n+1 }   \Delta^{r-m+k|1|s}(a_r, \ldots, a_{m+1}, \mu^{m-k}(a_m, \ldots, a_{k+1}), a_k,  \ldots,  \underline{p}, \ldots,  a_{|s}  ) = 0.
\end{multline}

In order to construct this operation, we imitate the construction of the operations $\mu^d$, replacing $\Discbar_d$ by the moduli space $  \Discbar_{r|1|s}$ discussed in Appendix \ref{sec:stash-polyh-contr}:
\begin{defin} \label{def:floer_datum_disc_2_outputs}
A \emph{Floer datum} $D_{T}$ on a stable disc $T \in   \Discbar_{r|1|s} $  consists of the following choices on each component:
\begin{enumerate}
\item Strip-like ends near the marked points:  We have a parametrisation $\epsilon^k \co Z_{\pm} \to T$ (or $\epsilon^{|k} $, or $\underline{\epsilon}$)   of each end $\xi^k  \in  \overline{T}$  (respectively $ \xi^{|k}$, or $\underline{\xi}$)  whose source is $Z_-$ if $k=-1, 0$ and $Z_+$ otherwise.  
\item Time shifting map:  A map $ \rho_{T} \co \partial \bar{T} \to (0,+\infty)$ which is constant near each end.  We write $w_{k,T}$, $\underline{w}_{T}$  or $w_{|k,T}$  for the value on the appropriate  end.
\item Basic $1$-form and Hamiltonian perturbations:  A closed $1$-form $\alpha_{T}$ whose restriction to the boundary vanishes and a map $H_{T} \co T \to \sH(M)$  on each surface defining a Hamiltonian flow $X_{T}$ such that the pullback of $ X_{T} \otimes \alpha_{T} $ under $ \epsilon^k$ agrees with  $  X_{\frac{H}{w_{k,T} }  \circ \psi^{w_{k,T}}} \otimes dt$ (and the corresponding condition near  $ \xi^{|k}$ and $\underline{\xi}$ also holds). 
\item Almost complex structures:  A map $I_{T} \co S \to \sJ(M)$ whose pullback under $ \epsilon^k $ agrees with  $ (\psi^{w_{k,T}})^{*} I_t $ (with the corresponding condition at  $ \xi^{|k}$ and $\underline{\xi}$).
\end{enumerate}
\end{defin}
As before, the closedness of the $1$-form implies that
\begin{equation*} w_{-1,T} +  w_{0,T}  = \sum_{1 \leq k \leq r} w_{k,T}  + \underline{w}_{T} +  \sum_{1 \leq k \leq s} w_{|k,T}. \end{equation*}

\begin{defin}
A  \emph{universal and conformally consistent}  choice of Floer data for the bimodule map $\Delta$ is a choice   $\bfD_{\Delta}$ of Floer data  for every pair of integers $r, s \geq 0$, and every (representative) of an element of  $\Discbar_{r|1|s}$ which varies smoothly over this compactified moduli space, such that the two Floer data on any irreducible component of a singular disc are conformally equivalent.  At each stratum of the boundary Floer data agree, to infinite order in the coordinates coming from strip like ends, with those obtained by gluing.
\end{defin}

\begin{figure}
  \centering
     \includegraphics{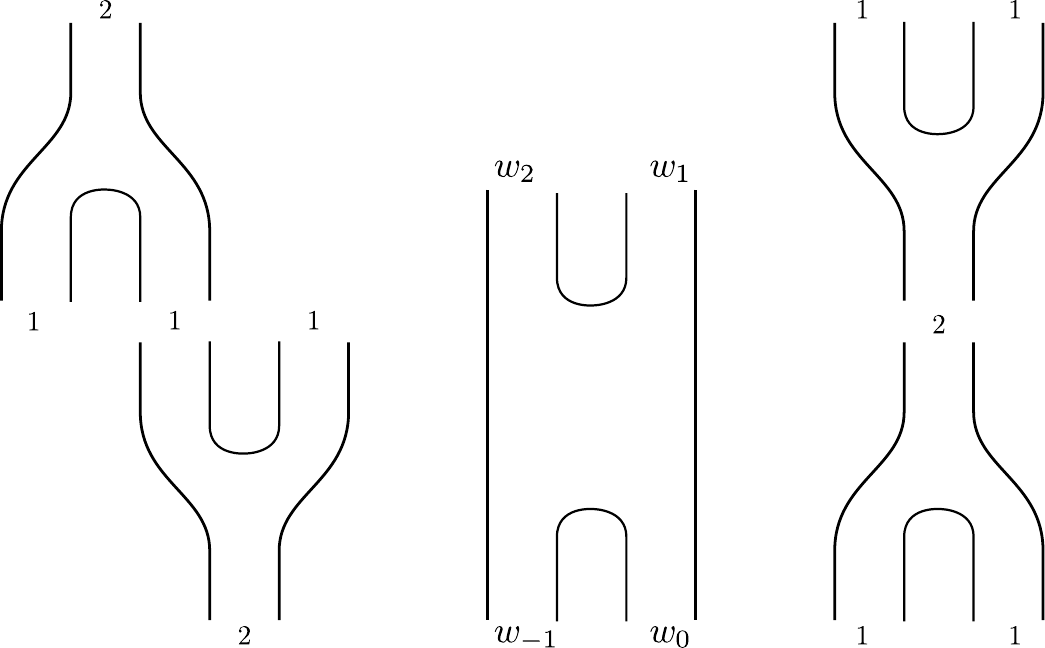}
  \caption{ }
  \label{fig:higher_coprod_data}
\end{figure}

\begin{rem}
  In the special case $r=s=0$ treated in Section \ref{sec:copr-wrapp-categ}, the weights at the two outputs were equal; this is false in general.  To see the necessity of unequal weights, consider the simplest case where $r=1$ but $s=0$.  The corresponding moduli space $  \Discbar_{1|1|0} $ is homeomorphic to a closed interval as illustrated in Figure \ref{fig:higher_coprod_data}, where the notation has been simplified relative to Figure \ref{fig:floer_data_three_punctures} by recording only the weights at the ends.  Near one endpoint, the weights correspond, up to rescaling, to the weights chosen when $r=s=0$ and hence can indeed be chosen equal.  However, near the other endpoint the reader can see that the weights are not equal for the conformal constants we have chosen, and is invited to verify that this asymmetry persists for all possible conformal constants.
\end{rem}

The different possible configurations of singular discs correspond to the strata of   the boundary of $ \Discbar_{r|1|s} $ listed in  Equation \eqref{eq:boundary_strata_stasheff_bimodules-1}-\eqref{eq:boundary_strata_stasheff_bimodules-5}.  We are requiring that the Floer data on $\Discbar_{r|1|s}$ be conformally equivalent  to the data $\bfD_{\mu}$ if an irreducible component comes from the factor $\Discbar_{d}$ for the appropriate $d$ or to the data $\bfD_{\Delta}$ for smaller values of $r$ and $s$.  In particular, having fixed the data  $\bfD_{\mu}$ so that transversality holds in Lemma \ref{lem:abundance_floer_data}, the existence of enough such universal data to guarantee transversality follows from the same argument as in Lemma \ref{lem:abundance_floer_data}.

Let $L_0, \ldots,  L_{r}$ and  $L_{|0}, \ldots , L_{|s}$ be sequences of  objects of $\cB$, and $K$ an object of $\Wrap$.   Consider chords $\vx = \{ x_{k} \in \Chord(L_{k-1},L_k) \}_{k=1}^{r}$,  $\underline{x} \in  \Chord(L_{|0},L_{0})$,  $\vx[|] = \{ x_{|k} \in \Chord(L_{|k},L_{|k-1}) \}_{k=1}^{s}$ as well as $x_{-1} \in \Chord(K, L_{r}) $ and  $x_0 \in \Chord(L_{|s}, K) $.   We associate these chords to the ends   $(\xi^{-1}, \xi^{|s} , \ldots, \xi^{|1}, \xi^0,  \underline{\xi}  , \xi^{1} , \ldots, \xi^{r}) $  of every disc $T \in \Disc_{r|1|s}$, and label the boundary components of $T$  with the appropriate Lagrangian.  Given universal Floer data $\bfD_{\Delta}$, we write $\Disc_{r|1|s}(x_{-1}, x_0 ;  \vx[|], \underline{x}, \vx) $ for the space of maps $u \co T \to  M$ which satisfy constraints along the boundary and the ends analogous to Equation \eqref{eq:boundary_condition_disc_A_infty},  and which solve the differential equation
\begin{equation}
  \label{eq:dbar_disc_2_outputs}
  \left(du - X_{T} \otimes \alpha_{T}\right)^{0,1} = 0
\end{equation}
where the $(0,1)$ part is taken with respect to the $T$-dependent almost complex structure.

Before introducing a morass of notation, it shall be helpful to recall that the boundary strata of any moduli space we will construct decompose as a product of lower dimensional moduli spaces; in most cases these can be associated to algebraic operations.  Whenever this is the case, we shall write the products in the \emph{reverse} order of the algebraic composition.

The consistency condition imposed on  $\bfD_{\mu}$ implies that the Gromov bordification  $\Discbar_{r|1|s}(x_{-1}, x_0 ; \vx[|], \underline{x}, \vx) $ is obtained by adding the following codimension $1$ strata
\begin{align*}
  \Discbar_{r-m+1}(x_{-1}; \vx[-1]_{m+1})  \times \Discbar_{m|1|s}(y_{-1}, x_0 ; \vx[|], \underline{x}, \vx[m]_{1})  & \quad  \begin{cases} y_{-1}  \in  \Chord(K, L_{m})  \\  \vx[m]_1 =  (x_{1}, \ldots, x_{m}) \\ \vx[-1]_{m+1} =(y_{-1}, x_{m+1} , \ldots, x_{r})  \end{cases} \\
 \Discbar_{s-\ell+1}(x_{0}; \vx[|\ell+1]_{0})  \times \Discbar_{r|1|\ell}(x_{-1}, y_0 ; \vx[|1]_{|\ell}, \underline{x}, \vx)  & \quad  \begin{cases} y_{0}  \in  \Chord(L_{|\ell}, K)  \\  \vx[|1]_{|\ell} =  (x_{|\ell}, \ldots, x_{|1}) \\ \vx[|\ell+1]_{0} =(x_{|s} , \ldots, x_{|\ell+1}, y_0)  \end{cases} \\
\Discbar_{r-m|1|s-\ell}(x_{-1}, x_0 ; \vx[|\ell+1]_{|s}, \underline{y}, \vx[r]_{m+1})  \times  \Discbar_{\ell+m+1}(\underline{y}; \vx[m]_{|\ell})  & \quad  \begin{cases} \underline{y} \in  \Chord(L_{m},L_{|\ell})  \\  \vx[r]_{m+1} =  (x_{m+1}, \ldots, x_{r}) \\ \vx[m]_{|\ell} =(x_{|\ell}, \ldots, x_{|1}, \underline{x}, x_{1} , \ldots, x_{m}) \\  \vx[|\ell+1]_{|s} = (x_{|s} , \ldots, x_{|\ell+1})  \end{cases} \\
\Discbar_{r|1|s-\ell+k+1}(x_{-1}, x_0 ; \vx[|]_{y}, \underline{x}, \vx)  \times  \Discbar_{\ell-k}(y; \vx[|k+1]_{|\ell})  & \quad  \begin{cases} y \in  \Chord(L_{|\ell},L_{|k})  \\  \vx[|]_{y} =  (x_{|s}, \ldots, x_{|\ell+1}, y, x_{|k}, \ldots, x_{|1})  \\ \vx[|k+1]_{|\ell} =(x_{|\ell}, \ldots, x_{|k+1}) \end{cases} \\
\Discbar_{r-m+k+1|1|s}(x_{-1}, x_0 ; \vx[|], \underline{x}, \vx_{y})  \times  \Discbar_{m-k}(y; \vx[m]_{k+1})  & \quad  \begin{cases} y \in  \Chord(L_{k},L_{m})  \\  \vx_{y} =  (x_1, \ldots, x_k, y, x_{m+1}, \ldots, x_{r})  \\ \vx[m]_{k+1} =(x_{k+1},  \ldots , x_{m}) \end{cases}
\end{align*}

In order to see that these cases cover all possible breakings of virtual codimension $1$, we first observe that one possibility is the breaking of a strip; if the breaking takes place at an input, the position relative to  $\underline{x}$ determines which of the last three strata accounts for it, and it otherwise appears in the first two strata.  

The remaining possibilities correspond to both components being stable; these should be labelled by the strata of the abstract moduli space $\Discbar_{r|1|s}$: Indeed, the first two cases correspond to the codimension $1$ strata (\ref{eq:boundary_strata_stasheff_bimodules-1}) and (\ref{eq:boundary_strata_stasheff_bimodules-2}), while the next three cases correspond to the remaining strata (\ref{eq:boundary_strata_stasheff_bimodules-3})-(\ref{eq:boundary_strata_stasheff_bimodules-5}).

Using Lemma \ref{lem:stay_in_compact_set} to prove compactness, and the usual Sard-Smale argument to prove transversality, we conclude:
\begin{lem} \label{lem:moduli_space_2_outputs_manifold_boundary}
For fixed inputs $( \vx[|], \underline{x}, \vx ) $, there are finitely many pairs $(x_{-1},x_0)$ such that  the moduli spaces $\Discbar_{r|1|s}(x_{-1}, x_0 ; \vx[|], \underline{x}, \vx) $  are non-empty.  Moreover, those which are not empty are compact and they form manifolds with boundary of dimension
\begin{equation} \label{eq:dimension_moduli_2_outputs}
  \deg(x_0) + \deg(x_{-1})-n + r+s - \deg(\underline{x}) - \sum_{1 \leq k \leq r} \deg(x_k) - \sum_{1 \leq k \leq s} \deg(x_{|s}) 
\end{equation} 
for generic data $\bfD_{\Delta}$.
\qed 
\end{lem}

Assuming that the expression in Equation \eqref{eq:dimension_moduli_2_outputs} vanishes,  the elements of  $ \Discbar_{r|1|s}(x_{-1}, x_0 ; \vx[|], \underline{x}, \vx)  $ are rigid. Using Equation (\ref{eq:orientation_coproduct}), and the orientation of $ \Disc_{r|1|s} $ fixed in Appendix \ref{sec:abstr-moduli-space}, we obtain an isomorphism
\begin{equation*}
o_{x_{r}} \otimes \cdots \otimes  o_{x_{1}} \otimes o_{\underline{x}} \otimes o_{x_{|1}} \otimes \cdots \otimes o_{x_{|s}}   \cong   o_{x_{-1}} \otimes  o_{x_{0}}.
\end{equation*}
Writing $\Delta_{u}$ for the induced map on orientation lines, we define 
\begin{equation} \label{eq:bimodule_map_signed_sum}
  \Delta^{r|1|s}  = \sum (-1)^{\ddagger} \Delta_{u}
\end{equation}
where the sum is taken over rigid elements of the moduli spaces $\Discbar_{r|1|s}(x_{-1}, x_0 ; \vx[|], \underline{x}, \vx)$, and the sign is given by
\begin{equation} \label{eq:ddagger_sign}
\ddagger = \sum_{j=1}^{s} (s-j+1) \deg(x_{|j}) + s \deg(\underline{x})  + \sum_{j=1}^{r} (j+s) \deg(x_j) .
\end{equation}
The proof the next result follows from matching the terms of Equation \eqref{eq:infinity_bimodule_equation} with the boundary strata listed before Lemma \ref{lem:moduli_space_2_outputs_manifold_boundary}.    The sign verification is simply a tedious extension of Lemma \ref{lem:Delta_chain_map}.
\begin{lem}
The operations $\Delta^{r|1|s}$ defined in Equation (\ref{eq:bimodule_map_signed_sum}) satisfy Equation (\ref{eq:infinity_bimodule_equation}), and hence define a map of bimodules. \qed
\end{lem}
\section{Maps relating open and closed strings} \label{sec:maps-relating-open}

\subsection{Symplectic cohomology}
\subsubsection{Breaking the $S^1$ symmetry}
Let $F \co S^1 \times M \to \bR$ be a smooth non-negative function such that
\begin{equation}
  \label{eq:support_condition_perturbation}
  \parbox{32em}{$F$ and  $ \lambda(X_F)$ are uniformly bounded in absolute value, and there is a sequence $R_i \to  +\infty$ such that $F(t, r, y) $ vanishes if $r$ lies in some open neighbourhood of $R_i$.}
\end{equation}

The choice of $F$ is meant to break the $S^1$-invariance of Hamiltonian orbits of $H$ by considering instead the orbits of
\begin{equation*}
  H_{S^1}(t,m) \equiv H(m) + F(t,m).
\end{equation*}
In particular, we write  $X_{S^1}$ for the time-dependent Hamiltonian vector field of  $H_{S^1}$ whose set of time-$1$ periodic orbits will be denoted $\Orbit$.
\begin{lem}
For a generic function $F$ satisfying Assumption \eqref{eq:support_condition_perturbation}, all time-$1$ periodic orbits  of $ X_{S_1} $ are non-degenerate.
\end{lem}
\begin{proof}
 Since time-$1$ Hamiltonian orbits of $H$ correspond to Reeb orbits of $\lambda | \partial M$, and since the length spectrum of Reeb orbits is discrete by Assumption \eqref{eq:non-degenerate_Reeb_orbit}, we may choose a sequence  $R_i \to  +\infty$  which does not intersect the length spectrum.  In particular, $F$ may be picked arbitrarily among $C^1$ bounded functions in a neighbourhood of the orbits of $X_{S^1}$, which implies non-degeneracy for generic choices of perturbations.
\end{proof}
Choosing conventions compatible with the ones for Lagrangians, we define the degree of $y \in \Orbit$ in terms of the Conley-Zehnder index as
\begin{equation}
  \label{eq:formula_degree_CZ}
\deg(y) = n - CZ(y) .
\end{equation}

Given an $S^1$-dependent family  $I_{S^1} \in  \sJ(M)$, we consider maps
\begin{equation*} u \co (-\infty, +\infty) \times S^1 \to M  \end{equation*}
converging exponentially at each end to a time-$1$ periodic orbit of $H_{S^1}$, and   satisfying  Floer's equation
\begin{equation}
  \label{eq:dbar_cylinder}
 \left(du - X_{S^1}  \otimes  dt \right)^{0,1} = 0
\end{equation}
with respect to the $S^1$-dependent almost complex structure $I_{S^1} $.  Unless $y_0=y_1$, we write $\Cyl(y_0,y_1)$ for the quotient by $\bR$ of the moduli space of solutions to the Cauchy Riemann equation \eqref{eq:dbar_cylinder} with asymptotic limits $y_0$ at $s=-\infty$ and $y_1$ at $s=+\infty$, and   $\Cylbar(y_0,y_1)$ for its Gromov bordification.  The analogue of Lemma \ref{lem:compactification_strip_manifold} holds:
\begin{lem}\label{lem:finite_number_outputs_cyl}
For a generic family $I_{S^1}$, the moduli space  $\Cylbar(y_0,y_1)$ is a compact manifold with boundary of dimension $\deg(y_0) - \deg(y_1) -1 $, whose boundary is covered by the closure of the natural inclusions
\begin{equation*}   \Cyl(y_0, y) \times   \Cyl(y, y_1) \to     \Cylbar(y_0, y_1) . \end{equation*} 
Moreover, for each orbit $y_1$, the moduli space $\Cylbar(y_0,y_1)$ is empty for all but finitely many choices of $y_0$.
\end{lem}
\begin{proof}
The perturbation $F$ prevents us from directly applying the maximum principle to prove compactness.  However, choosing $R_i$  sufficiently large, Lemmas \ref{lem:stay_in_compact_set} and \ref{lem:action_negative} imply that  all elements of $ \Cyl(y_0, y_1) $ are contained within the compact domain $r \leq R_i$.  
\end{proof}

\subsubsection{The symplectic chain complex}
 Consider the graded abelian group $SC^*$ with components
\begin{equation}
 \label{eq:symplectic_complex}
 SC^{i}(M) \equiv \bigoplus _{\stackrel{y \in \Orbit}{  \deg(y) = i}} | o_y|.
\end{equation}
Here, $| o_y|$ is again the orientation line on a rank-$1$ real vector space $o_y$ associated to each periodic orbit in Definition \ref{def:orientation}.

The differential on the symplectic chain complex is defined analogously to the definition of the differential in wrapped Floer cohomology as a sum of contributions of rigid cylinders $u \in \Cylbar(y_0,y_1)$ :
\begin{align}
\notag  \partial \co  SC^{i}(M)  & \to SC^{i+1}(M)  \\  \label{eq:symplectic_differential}
[y_1] & \mapsto (-1)^{i} \sum_{u} \partial_u([y_1] ),
\end{align}
where $\partial_u$ is the map induced on orientation lines by the isomorphism coming from Equation \eqref{eq:orientation_differential_SH}.  The finiteness of the right hand side this time follows from Lemma \ref{lem:finite_number_outputs_cyl}.  The cohomology of this complex will be denoted $SH^*(M)$, and is called \emph{symplectic cohomology}.

\subsection{From the closed to the open sector}

In this section, we define a chain map
\begin{equation}
  \label{eq:closed_to_open}
\CO \co  SC^*(M) \to CW^{*}(K,K)
\end{equation}
for every object $K$ of $\Wrap$.   The reader should keep in mind the surface at the top left of Figure \ref{fig:Cardy}.

Let $T_{1}^{1}$ denote the complement of an interior marked point $\sigma$ and a boundary marked point $\xi^0$ on a disc.  Up to biholomorphism, we may assume that we are considering the unit disc with $\sigma$ the origin and $\xi^0$ equal to $1$. Using polar coordinates, we fix the positive cylindrical end
\begin{align}
\label{eq:cylindrical_end_positive_puncture} \epsilon^{1} \co S^1 \times [0,+\infty) & \to T_{1}^{1} \\
\epsilon^{1}(\theta, r) & = e^{-(r+ i \theta) + i \pi} 
\end{align}
near $\sigma$, a negative strip-like end $\epsilon^0$ near $\xi^0$, and a closed $1$-form $\alpha_{1}^{1}$ whose pullback under $\epsilon^{k}$ agrees with $dt$ and whose restriction to the boundary vanishes.  In addition, choose $1$-forms $\beta^{1}_1$ and $\gamma_{1}^{1}$ whose restrictions to the boundary also vanish, and whose pullbacks under $\epsilon^1$ agree respectively with $dt$ and $0$  while their pullbacks under $\epsilon^0$ are $0$ and $dt$.  We require that the following conditions hold
\begin{align*}
\begin{cases}  \alpha_{1}^{1} \wedge \beta_{1}^{1} =   \alpha_{1}^{1} \wedge \gamma_{1}^{1} = \beta_{1}^{1} \wedge \gamma_{1}^{1} = 0 \\ 
d \gamma_{1}^{1}  \leq  0 \\
\overline{ \{ d  \beta_{1}^{1}  \neq 0 \}}   \subset \{ d  \gamma_{1}^{1}  \neq 0 \}.  \end{cases}
\end{align*}
Finally, we pick a map  $I_{1}^{1} \co T_{1}^{1} \to \sJ(M)$ whose pullback under $\epsilon^{k}$ agrees with $(\psi^{2})^{*}I_{t}$ if $k=0$ and $I_{S^1}$ if $k=1$. 

Given a time-$1$ orbit $y_1$ of $X_{S^1}$, and a time-$1$ chord $x_{0}$  with both endpoints on $K$,  we write $\Disc_{1}^{1}(x_0;y_1)$ for the space of maps $u$  from $T_{1}^{1}$ to $M$ with the following boundary conditions
\begin{equation}
  \label{eq:boundary_disc_1_interior_1_boundary_output}
  \begin{cases}  u(z)  \in \psi^{2} K & \textrm{if $z \in \partial T_{1}^{1} $ } \\
\lim_{s \to - \infty} u \circ \epsilon^{0}(s, \cdot) = \psi^{2} \circ x_0(\cdot) &  \\
\lim_{s \to + \infty} u \circ \epsilon^{1}(s, \cdot) = y_1(\cdot)
 \end{cases}
\end{equation}
and solving the differential equation
\begin{equation}
  \label{eq:dbar_map_from_SH}
  \left( du -   X_{H}   \otimes \alpha_{1}^{1} -  X_{F} \otimes \beta^{1}_{1} -   \left( X_{ \frac{H}{2} \circ \psi^{2}}  - X_{H}  \right) \otimes  \gamma^{1}_1  \right)^{0,1} = 0.
\end{equation}
The key point is that the pullback of the above equation under $\epsilon^1$ agrees with Equation \eqref{eq:dbar_cylinder}, while its pullback under $\epsilon^0$  agrees with Equation \eqref{eq:dbar_strip} up to the  symplectically conformal map $\psi^{2}$.    In proving compactness, we simply note that $ \frac{H}{2} \circ \psi^{2} -H$ agrees with $r^{2}$ on the cylindrical end, which implies that away from a compact subset of $M$, $\left(   \frac{H}{2} \circ \psi^{2} -H \right)  d \gamma_{1}^{1}$ agrees with $r^2 d\gamma_{1}^{1}$  and hence dominates $ F  d \beta_{1}^{1}$ since $F$ is uniformly bounded.   In particular, the hypothesis of  Lemma \ref{lem:stay_in_compact_set} holds if $\{ R_i \} \times \partial M  $ separates $y_1$ from infinity, and we let $S =u^{-1}([R_i,+\infty) \times \partial M)$.  Applying the usual transversality argument, we conclude
\begin{lem}
For a generic family $I_{1}^{1}  $, the Gromov bordification of the moduli space $\Disc_{1}^{1}(x_0;y_1)$ is a compact manifold with boundary of dimension
\begin{equation*}
  \deg(x_0) - \deg(y_1)
\end{equation*}
which, for fixed $y_1$ is empty for all but finitely many chords $x_0$.  The codimension $1$ boundary strata are the images of the natural inclusions
\begin{equation}
  \label{eq:boundary_disc_from_SH}
\coprod \Disc(x_0;x_1) \times  \Disc_{1}^{1}(x_1;y_1)   \cup \coprod  \Disc_{1}^{1}(x_0;y_0) \times \Cyl(y_0;y_1)   \to \partial  \Discbar_{1}^{1}(x_0;y_1). 
\end{equation} \qed
\end{lem}

The outcome is that whenever $  \deg(x_0)  = \deg(y_1) $ the moduli space $\Disc_{1}^{1}(x_0;y_1)$ is rigid, so every element $u$ gives an isomorphism
\begin{equation*}
  o_{y_1} \to   o_{x_0},
\end{equation*}
coming from Equation \eqref{eq:orientation_SH_to_HW}, and hence a map $\CO_u$  on orientation lines.  We define 
\begin{equation}
  \label{eq:SH_to_HF_formula}
  \CO = \sum \CO_{u}
\end{equation}
with the sum taken over all rigid discs.   The argument for the validity of the signs is essentially the same as the one for the proof that the differential in the wrapped Floer complex squares to $0$, and hence is omitted.
\subsection{From the open to the closed sector}
In Appendix \ref{sec:discs-with-one}, we construct a moduli space  $\Disc_{d}^1$  of abstract discs with one interior outgoing end, and several boundary end. The main result of this section is:
\begin{lem} \label{lem:open_to_closed}
The counts of solutions to $\dbar$ operators parametrised by the  moduli spaces $\Disc_{d}^1$ define maps
\begin{equation} \label{eq:open_closed_length_d} \OC_{d} \co  CW^{*}(L_{d-1},L_0) \otimes \cdots \otimes CW^{*}(L_1, L_2) \otimes CW^{*}(L_0, L_1)   \to SC^{*}(M)   \end{equation} 
which shift degree by $n-d+1$ and are the components of a degree $n$ chain map
\begin{equation}
\OC \co CC_{*}(\cB, \cB) \to SC^{*}(M)
\end{equation}
where the left hand-side is the cyclic bar complex of $\cB$ equipped with the differential computing Hochschild homology.
\end{lem}

Recall that to each $A_{\infty}$ category, one may assign a Hochschild homology group which is the homology of the cyclic bar complex with respect to the differential
\begin{multline}
  \label{eq:hochschild_differential}
  b (a_{d} \otimes \cdots \otimes a_{1}) =  \sum_{ i+j \leq d} (-1)^{\mathsection} \mu^{d-j-1}( a_{i-1}, \ldots, a_1, a_d, \ldots, a_{i+j+1} ) \otimes a_{i+j} \otimes \ldots \otimes a_{i}   \\
+ \sum_{ i+j \leq d} (-1)^{\maltese_{1}^{i-1}} a_{d} \otimes \cdots \otimes a_{i+j+1} \otimes \mu^{j+1}(a_{i+j} , \ldots, a_{i}) \otimes a_{i-1} \otimes \ldots \otimes a_{1}
\end{multline}
where the first sign is given by
\begin{equation}
\mathsection = \maltese_{1}^{i-1} \cdot (1+ \maltese_{i}^{d}) + \maltese_{i}^{d-1}+1.
\end{equation}
\begin{rem}
Note that our cyclic bar complex is equipped with a cohomological differential, such that the degree of $a_{d} \otimes \cdots \otimes a_{1} $ given by
\begin{equation}
  \label{eq:degree_hochschild_cplex}
|| a_{d} \otimes \cdots \otimes a_{1} || =   \deg(a_{d}) + \sum_{i=1}^{d-1} ||a_{i}||.
\end{equation}
\end{rem}

The simplest of these operations ($d=1$) maps the wrapped Floer complex of a Lagrangian $L$ to the symplectic chain complex.  The labelled curve controlling it appear at the top right of Figure \ref{fig:Cardy}.  By now, the reader should be quite familiar with our method of constructing operations from moduli spaces, so instead of illustrating the case $d=1$, we jump directly to the general one.

\begin{defin} \label{def:floer_datum_discs_interior_puncture}
A \emph{Floer datum} $D_{S}$ on a stable disc $S \in \Discbar_{d}^{1}$ with $d$ positive boundary punctures $(\xi^1, \ldots, \xi^d) $ and one negative interior puncture $\sigma$ consists of the following choices on each component:
\begin{enumerate}
\item Strip-like ends near the punctures:  we have  maps $\epsilon^k \co Z_{+} \to S$ for each $1 \leq k \leq d$ converging at the end of the half-strip $Z_{+}$ to $\xi^{k}$.   In addition, we fix a cylindrical end $\epsilon^0 \co (-\infty,0] \times S^1 \to S$ for $\sigma$ which extends to a biholomorphism $(-\infty, 1] \times S^1 \to \overline{S} - \sigma$ taking $(1,1)$ to $\xi^{d}$.
\item Time shifting map:  A map $ \rho_{S} \co \partial \bar{S} \to [1,d]$ which is constant near each marked point.  We write $w_{k,S}$ for the value on the $k$\th end and set  
  \begin{equation}
    w_{0,S} = \sum_{k=1}^{d} w_{k,S}
  \end{equation}
\item Basic $1$-form and Hamiltonian perturbations:  A closed $1$-form $\alpha_{S}$ whose restriction to the boundary vanishes and a map $H_{S} \co S \to \sH(M)$  on each surface such that the pullback of $ X_{H_S} \otimes \alpha_{S} $ under $ \epsilon^k $ agrees with  $  X_{\frac{H}{w_{k,S}} \circ \psi^{w_{k,S}}} \otimes dt$.
\item Sub-closed $1$-form: A $1$-form $\beta_{S}$ which may be written as the product of a smooth function with $\alpha_{S}$,  satisfying $d \beta_{S} \leq 0$,  and whose pullback under $\epsilon^k$ vanishes unless $k=0$, in which case it agrees with $dt$.
\item Almost complex structures:  A map $I_{S} \co S \to \sJ(M)$ whose pullback under $ \epsilon^k $ agrees with  $ (\psi^{w_{k,S}})^{*} I_{t} $ unless $k=0$ in which case it agrees with  $ (\psi^{w_{0,S}})^{*} I_{S^1}$. 
\end{enumerate}
\end{defin}

\begin{defin}
 \emph{Universal and conformally consistent} Floer data for the map $\OC$ consist of a choice $\bfD_{\OC}$ of Floer data  for every integer $d \geq 1$ and every (representative) of an element of  $\Discbar_{d}^1$ which varies smoothly over the compactified moduli space. The two natural Floer data (coming from $\bfD_{\OC}  $ or $ \bfD_{\mu}$)  on any irreducible component of a singular disc should be conformally equivalent, and the Floer data should agree to infinite order near each boundary stratum with those obtained by gluing.
\end{defin} 
In particular, Floer data have been chosen for integers less than $d$, the data on the boundary of $\Discbar_{d}^1$ is determined, via the maps described in the discussion surrounding Equation \eqref{eq:boundary_discs_1_puncture}, by the choice of Floer data $ \bfD_{\mu}$.

Consider a collection of Lagrangians $L_0, L_1, \ldots , L_d$ with $L_d = L_0$  and a sequence of chords $\vx = \{ x_1, \ldots, x_d \}$ with  $x_k \in \Chord(L_{k-1},L_k)$ for $1 \leq k \leq d$, as well as an orbit $y_0 \in \Orbit$.  We define  $ \Disc_{d}^1(y_0; \vx) $ to be the moduli space of maps $u \co S \to M$, with $S$ an arbitrary element of $ \Disc_{d}^1 $, satisfying the boundary and asymptotic conditions
\begin{equation}
\begin{cases}
u(z) \in \psi^{\rho_{S}(z)} L_k & \textrm{if $z \in \partial S$ lies between $\xi^k$ and $\xi^{k+1}$} \\
\lim_{s \to \pm \infty} u \circ \epsilon^{k}(s, \cdot) =  \psi^{w_{k,S}} x_k(\cdot)  & \textrm{if $1 \leq k \leq n$} \\
\lim_{s \to \pm \infty} u \circ \epsilon^{0}(s, \cdot) =  \psi^{w_{0,S}} y_0(\cdot) &
\end{cases}
\end{equation}
and solving the differential equation
\begin{equation}
  \label{eq:dbar_disc_1_interior_output}
\left(du - X_{H_S} \otimes \alpha_{S} - X_{\frac{F}{w_{0,S}} \circ \psi^{w_{0,S}}} \otimes \beta_{S} \right)^{0,1} = 0
\end{equation}
where the $(0,1)$ part is taken with respect to the $S$-dependent almost complex structure, and the function $F$ is the one appearing in the definition of symplectic cohomology, see, e.g. Condition \eqref{eq:support_condition_perturbation}.

The consistency condition imposed on  $\bfD_{\OC}$ implies that the Gromov bordification $\Discbar_{d}^1(y_0; \vx) $  is obtained by adding the images of natural inclusions
\begin{align}
  \label{eq:codim_1_strata_discs_1_interior_output}
  \Cyl(y_0; y_1) \times \Discbar_{d}^1(y_1; \vx) & \to \partial \Discbar_{d}^1(y_0; \vx)  \\  \label{eq:codim_1_strata_discs_1_interior_output_2}
  \Discbar_{d}^1(y_0; \vx[1])  \times   \Discbar_{d_2}(x; \vx[2]) & \to \partial \Discbar_{d}^1(y_0; \vx) 
\end{align}
where in the second type of stratum, $x$ agrees with one of the elements of $\vx[1] $ and the sequence obtained by removing $x$  from $\vx[1]$ and replacing it by the sequence $\vx[2]$ agrees with $\vx$ up to cyclic ordering.   Since $d\beta_{S}$ is semi-negative and $F$ is positive,
Lemma \ref{lem:stay_in_compact_set} easily applies to this moduli space so we conclude
\begin{lem}
The moduli spaces $\Discbar_{d}^1(y_0; \vx) $  are compact and there are only finitely many choices of orbits $y_0$ for which they are non-empty once the input sequence $\vx$ is fixed.  In addition, for generic universal and conformally consistent Floer data $\bfD_{\OC}$ they form manifolds of dimension
\begin{equation*}
  \deg(y_0) - n + d  - 1  - \sum_{1 \leq k \leq d} \deg(x_k).
\end{equation*}\qed
\end{lem}

Assuming that $\deg(y_0) =  n -d + 1 + \sum_{1 \leq k \leq d} \deg(x_k)  $, we conclude that the elements of  $\Disc_{d}^1(y_0; \vx) $ are rigid.  Orienting the tangent space of $ \Disc_{d}^1$ using Equation \eqref{eq:orientation_discs_1_marked_point} and applying Equation \eqref{eq:orientation_HH_to_SH}, we obtain for each disc $u$ an isomorphism
\begin{equation}
  o_{x_d} \otimes \cdots \otimes o_{x_1} \to o_{y_0}.
\end{equation}
Writing $\OC_{u}$ for the induced map on orientation lines, we define  the map $\OC_d$ in Equation \eqref{eq:open_closed_length_d} to be the sum
\begin{equation}
  \OC_d([x_d], \ldots, [x_1])  = \sum_{\stackrel{\deg(y_0) =n -d + 1 + \sum_{1 \leq k \leq d} \deg(x^k) }{u \in \Disc_{d}^{1}(y_0, \vx) }} (-1)^{\deg(x_d) + \dagger} \OC_{u}( [x_d], \ldots, [x_1])
\end{equation}
where $\dagger$ is given in Equation \eqref{eq:dagger_sign}.

By comparing the boundary strata in Equation \eqref{eq:codim_1_strata_discs_1_interior_output} and \eqref{eq:codim_1_strata_discs_1_interior_output_2} with the formula for the differential $b$ on the cyclic bar complex, we conclude that Lemma \ref{lem:open_to_closed} holds up to signs.  To check it in the simplest situation, we let $d=1$, and note that in this case $\dagger  = \deg(x_d) $, so that $\OC_{1}$ agrees with the sum of the maps $\OC_u$.    The boundary of $\Disc_{1}^1(y_0; x_1)  $ is covered by
\begin{align*}
&  \Disc_{1}^1(y_0, x_0) \times    \Disc( x_0; x_1 ) \\
&  \Cyl(y_0; y_1) \times     \Disc_{1}^1(y_1; x_1) .
\end{align*}
In the first case, the boundary orientation is given by an isomorphism

\begin{equation*}
  o_{x_1}  \cong   \lambda^{-1}( \Disct(x_0;x_1) )   \otimes \lambda(\Disc_{1}^{1}) \otimes  o_{y_{0}} \otimes  \lambda(L_0).
\end{equation*}
Keeping in mind that the translation by $\partial_s$  agrees with an inward pointing normal vector after gluing, we conclude that the  restriction of  $\OC \circ b$ to $CC_1(\cB,\cB)$, which is given by  $- \OC \circ \mu_1$ agrees with the map induced at the boundary up to a sign whose parity is
\begin{equation}\label{eqsign_check_OC-1}
  1 + \deg(x_1).
\end{equation}
On the other boundary component, 

\begin{equation*}
  o_{x_1} \cong \lambda(\Disc_{1}^{1}) \otimes    \lambda^{-1}( \Cylt(y_0;y_1) )   \otimes  o_{y_0}  \otimes  \lambda(L_0) .
\end{equation*}
So that $ \partial \circ \OC $ agrees with the map induced at the boundary up to a sign of parity
\begin{equation}
  \deg(y_0) = \deg(x_1) + n.
\end{equation}
Comparing this with Equation \eqref{eqsign_check_OC-1}, we conclude that $ \OC_{1}$ is indeed a degree $n$ chain map.
\section{The Cardy relation} \label{sec:cardy-relation}

The next few sections construct the appropriate chain-level models for the maps described in Equations \eqref{eq:map_HH_tensor_over_A} and \eqref{eq:multiplication_right_left}, before giving the proof of Proposition \ref{prop:cardy_relation}.

\subsection{Algebraic preliminaries}\label{sec:algebr-prel}
Let $\cL$ and $\cR$ be respectively left and right $\cB$ modules.  The tensor product of $\cR$   and $\cL$ over $\cB$ is defined to be the chain complex
\begin{equation}
  \label{eq:tensor_over_A}
\bigoplus_{L_0, \ldots, L_{d} \in \Ob(\cB)}   \cR(L_{d}) \otimes CW^{*}(L_{d-1}, L_{d}) \otimes \cdots \otimes CW^{*}(L_{0}, L_{1})  \otimes \cL(L_0) 
\end{equation}
with grading
\begin{equation*}
  \label{eq:degree_tensor_over_A}
  \deg(p \otimes a_d \otimes \ldots \otimes a_1 \otimes q) = \deg(p) + \sum ||a_{i}|| + \deg(q)
\end{equation*}
and differential
\begin{multline*}
  p \otimes a_d \otimes \ldots \otimes a_1 \otimes q \mapsto \sum p \otimes  a_{d} \otimes \cdots \otimes  a_{\ell+1} \otimes \mu^{\ell|1}(a_{\ell}, \ldots, a_{1}, q) \\
+  \sum (-1)^{\deg(q) + \maltese_{1}^{\ell}} \mu^{1|d-\ell} (p ,  a_{d},  \ldots,   a_{\ell+1}) \otimes a_{\ell} \otimes  \cdots \otimes a_{1} \otimes q \\
+  \sum (-1)^{\deg(q) + \maltese_{1}^{\ell}}p  \otimes a_{d} \otimes \cdots \otimes a_{\ell+k+1} \otimes \mu^{k}(a_{\ell+k},   \ldots,   a_{\ell+1}) \otimes a_{\ell} \otimes  \cdots \otimes a_{1} \otimes q 
\end{multline*}

The map induced by $\Delta$ at the level of Hochschild chains is given by
\begin{multline}
  \label{eq:hochschild_map_induced_by_coproduct}
CC_{*}(\Delta) (a_{d} \otimes \ldots \otimes a_{1}) = \\ \sum (-1)^{\diamond} \sT\left( \Delta^{r|1|s}(a_{r}, \ldots, a_{1}, \underline{a}_{d}, a_{d-1}, \ldots, a_{d-s}) \otimes a_{d-s-1} \otimes \cdots \otimes a_{r+1} \right)
\end{multline}
where  $\sT$ is the maps which reorders the factors
\begin{equation*}
\sT( q \otimes p \otimes a_{d-s-1} \otimes \cdots \otimes a_{r+1}) = (-1)^{\circ} p \otimes a_{d-s-1} \otimes \cdots \otimes a_{r+1} \otimes q
\end{equation*}
and the signs are given by the formulae
\begin{align}
\label{eq:sign_delta_induce_cyclic} \diamond & = \maltese_{1}^{r} \cdot (1+ \maltese_{r+1}^{d})  + n \maltese_{r+1}^{d-s-1}   \\
  \label{eq:sign_reorder} \circ & = \deg(q) ( \deg(p)  + \maltese_{r+1}^{d-s-1}) .
\end{align}
We define $HH_{*}(\Delta) $  to be the map induced by  $ CC_{*}(\Delta) $  on homology groups.  Note that since $\Delta$ is a chain map of degree $n$, so is $CC_{*}(\Delta)$.

The remaining map in the diagram \eqref{eq:cardy_relation} comes from the chain-level composition
\begin{align}
  \label{eq:composition_left_right}
\mu \co \cY^{r}_{K} \otimes_{\cB} \cY^{l}_{K} & \to  HW^{*}(K,K) \\
p \otimes a_{d} \otimes \cdots \otimes a_{1} \otimes q & \mapsto (-1)^{\deg(q) + \maltese_{1}^{d}}  \mu^{d+2}(p, a_{d}, \ldots, a_{1}, q)
\end{align}
which may be easily verified to be a chain map (of degree $0$).

\subsection{Construction of the first homotopy}
In order to prove Proposition \ref{prop:cardy_relation}, we must construct a homotopy between the two compositions in the following square:
\begin{equation}
  \label{eq:chain_level_cardy_relation}
 \xymatrixcolsep{5pc}  \xymatrix{ CC_{*}(\cB, \cB) \ar[r]^{CC_{*}(\Delta)}  \ar[d]^{\OC} &  \cY^{r}_{K} \otimes_{\cB} \cY^{l}_{K} \ar[d]^{\mu} \\
SC^{*}(M) \ar[r]^{\CO}  & CW^{*}(K,K).}
\end{equation}

In particular, we must construct a map
\begin{equation}
  \label{eq:homotopy_from_annuli}
  \cH \co  CC_{*}(\cB, \cB)  \to  CW^{*}(K,K)
\end{equation}
satisfying
\begin{equation}
  \label{eq:homotopy_equation}
  (-1)^{n} \mu^{1} \circ \cH + \cH \circ b + \mu \circ CC_{*}(\Delta) - \CO \circ \OC = 0.
\end{equation}

Unfortunately, we shall have to construct this homotopy in two steps, because the compatibility conditions between the Floer data $\bfD_{\Delta} $ and $\bfD_{\mu}$  were meant to ensure that $\Delta$ defines a map of $A_{\infty}$-bimodules, but did not incorporate any other restrictions, so that while the composition $ \mu \circ CC_{*}(\Delta) $ can be interpreted as a count of broken curves, these cannot  be glued in general.  The problem does not occur for the simplest possible gluing (the bottom configuration in Figure \ref{fig:Cardy}), but already appears in the next simplest situations in Figure \ref{fig:gluing_problem} which shows two sets of broken curves, with weights labelling the ends of each component.  To keep the figures uncluttered, we have not tried to match the ends that need to be glued.  Note that on the right side, gluing can only be performed if $w_0=w_{-1}$ while on the left $w_1 = w_3$ is required.  It should be apparent from Figures \ref{fig:floer_data_three_punctures} and \ref{fig:higher_coprod_data} that these conditions cannot hold in general.  Moreover, changing the conformal constants will not resolve this problem.

\begin{figure}
  \centering
     \includegraphics{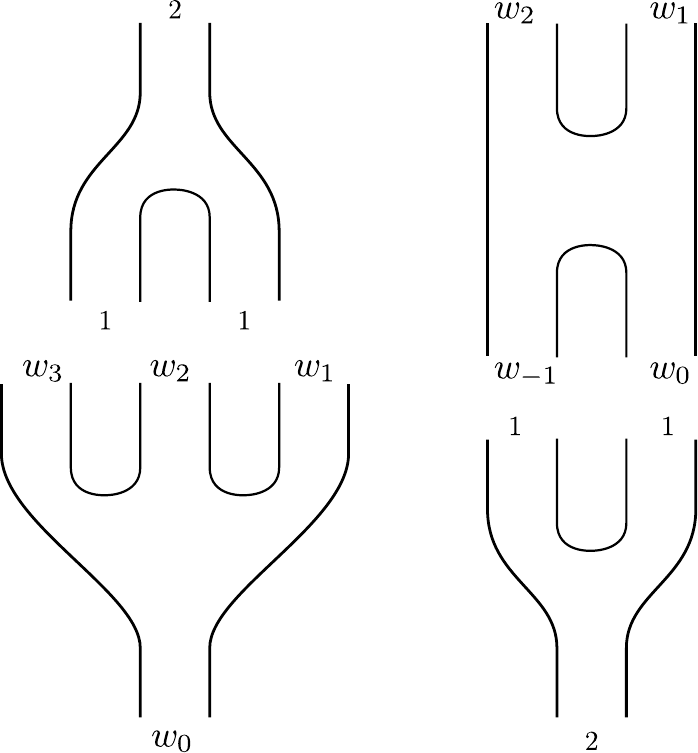}
  \caption{ }
  \label{fig:gluing_problem}
\end{figure}

Our first task is therefore to show that $ \mu \circ CC_{*}(\Delta) $  is homotopic to a map counting broken curves for which we may apply the usual gluing theory.  We shall do this by considering a family where the weights are allowed change.  The endpoint of this homotopy is a map $ CC_{*}(\cB, \cB)  \to  CW^{*}(K,K) $ which has no reason to factor on the nose through  $\cY^{r}_{K} \otimes_{\cB} \cY^{l}_{K}$.  Not having to factor through this space gives us the extra flexibility to correct the weight as necessary.  The space parametrising the curves counted in this homotopy is described in Section \ref{sec:an-auxiliary-moduli} and is of the form:
\begin{equation*}
  \Product_{d} = [0, 1] \times \bigcup_{r+s+d_1 - 1 =  d} \Discbar_{d-r-s+1} \times \Discbar_{r|1|s} / \sim
\end{equation*}

\begin{defin} \label{def:floer_datum_product_moduli_space}
A \emph{Floer datum} $D_{S_1}(t,S_2)$ on an element $(t, S_1 ,  S_2)$  of $\Product_{d}$ consists of a Floer datum on $S_1 \in  \Discbar_{d-r-s+1}$ in the sense of Definition \ref{def:floer_datum_disc_1_output}.  We write $H_{S_1}(t,S_2)$, $w_{k,S_1}(t,S_2)$ and so on for the maps and constants that constitute $D_{S_1}(t,S_2)$.

A   \emph{universal and conformally consistent}  choice of Floer data for the first homotopy is a choice  $\bfD_{\cH^1}$ of Floer data  for every integer $d \geq 1$, and every (representative) of an element of  $\Product_{d}$ which varies smoothly over the compactified moduli space, and is compatible with  $\bfD_{\mu}$  in the following sense:
\begin{enumerate}
\item The restriction of $\bfD_{\cH^1}$  to $(1, S_1 , S_2)$  agrees with $\bfD_{\mu} $ on $S_1$. 
\item The restriction of $\bfD_{\cH^{1}}$ to a component of $S_1$ which does not contain the outgoing end is conformally equivalent to the restriction of $\bfD_{\mu}$.
\item The two Floer data on a disc lying on the boundary stratum \eqref{eq:boundary_product_disc_bubbling} are conformally equivalent.
\end{enumerate}
In addition, the ratio between the weights on the first and last incoming ends for the restriction of $\bfD_{\cH^1}$  to $(0, S_1 ,  S_2)$ agrees with the ratio of the weights  coming from $\bfD_{\Delta}$ on the  corresponding outgoing ends of $S_2$: 
\begin{equation} \label{eq:ratio_weights_equal}
\frac{ w_{1,S_1}(0,S_2)}{ w_{d-r-s+1,S_1}(0,S_2)} =   \frac{ w_{0,S_2}}{ w_{-1,S_2}}.
\end{equation}

We write $\bfD_{f}$ for the restriction of $\bfD_{\cH^1}$ to $\Product_{d}^{0}$.
\end{defin}
As usual, the construction of universal and consistent data may be done inductively.  One starts by choosing Floer data whenever $d=1$ and $t=0$; in this situation, $S_{1}$ is necessarily a disc with two inputs.  Except for Condition \eqref{eq:ratio_weights_equal}, the choice of the data on such a disc is arbitrary.  The first consistency condition in Definition \ref{def:floer_datum_product_moduli_space} determines a choice of Floer data whenever $t=1$, and we interpolate arbitrarily between these two choices.

Assuming a choice of data has been fixed for integers smaller than $d$, we note that the third condition in Definition \ref{def:floer_datum_product_moduli_space} determines a choice of Floer data on the closure of the locus of $\partial \Product_{d}$ where $t \neq 0,1$.  On the intersection of this stratum with the set $t =0$, Condition \eqref{eq:ratio_weights_equal} holds by induction.  We first extend the choice of Floer data on the remainder of the stratum $t=0$ in such a way that Condition \eqref{eq:ratio_weights_equal} still holds, then interpolate between this choice and the one fixed  by the first consistency condition at $t=1$.  This provides an inductive construction of Floer data $\bfD_{\cH^1}$.

Consider a collection of Lagrangians $L_0, L_1, \ldots,  L_d$ with $L_d = L_0$  and a sequence of chords $\vx = \{ x_1, \ldots, x_d \}$ with  $x_k \in \Chord(L_{k-1},L_k)$ for $1 \leq k \leq d$, as well as a chord $x_0 \in \Chord(K,K)$.  We define the moduli space $ \Product_{d}(x_0; \vx) $ to be the union,  for all positive integers $r$ and $s$ such that $0 \leq r+s <d$, and over all pairs of chords
 \begin{equation*} (y_{-1}, y_0 ) \in \Chord(K, L_{r}) \times \Chord(L_{d-s-1} ,K)  \end{equation*}
and $t \in [0,1]$ of the moduli space of maps $(u_1,u_2) \co S_1 \cup S_2 \to M$ with  $u_2 \in \Discbar_{r|1|s}(y_{-1}, y_0 ; \vx[r]_{1}, \underline{x_d}, \vx[d-1]_{d-s}) $ and $u_1$ is a solution to Equation \eqref{eq:dbar_pair_pants} for the Floer datum $D_{S_1}(t,S_2)$ with asymptotic data $(y_{-1}, \vx[d-s-1]_{r+1}, y_{0})$ at the inputs, and $x_0$ at the output.  
\begin{rem}
Since the Floer datum on $ S_1 $  depends on $S_2$, there is no natural way to describe   $  \Product_{d}(x_0; \vx)  $  as a union of  products of the moduli spaces $  \Discbar_{r|1|s}(y_{-1}, y_0 ; \vx[r]_{1}, \underline{x_d}, \vx[d-1]_{d-s}) $  with another moduli space.
\end{rem}

\begin{lem} \label{lem:product_moduli_space_manifold}
For generic Floer data  $\bfD_{\cH^1}$, the moduli space $  \Product_{d}(x_0, \vx)  $  is a compact manifold with boundary of dimension
\begin{equation*}
 \deg(x_0) - n + d  - \sum_{k=1}^{d} \deg(x_k).
\end{equation*}\qed
\end{lem}

Writing $ \Product_{d}^{0}(x_0; \vx)$ for the subset of $ \Product_{d}(x_0; \vx) $ where $t=0$, we find that  the consistency conditions imposed on  $\bfD_{\cH^1}$ imply that the boundary of $\Product_{d}(x_0; \vx) $  is covered by the closures of the images of natural inclusions of the following moduli spaces:
\begin{align} \notag
\Product_{d}^{0}(x_0; \vx)  & \\ \notag
\Disc_{r|1|s}(y_{-1}, y_0 ; \vx[r]_{1}, \underline{x_d}, \vx[d-1]_{d-s})  \times \Disc_{d-r-s+1}(x_0; y_{-1}, \vx[d-s-1]_{r+1}, y_{0}) & \quad (y_{-1}, y_0 ) \in \Chord(K, L_{r})  \\ \notag
\Disc(x_0; x) \times \Product_{d}(x; \vx)    & \quad x  \in \Chord(K, K)  \\
\Product_{d_1}(x_0; \vx_{1}) \times  \Disc_{d-d_1+1}(x; \vx_2)  & \quad x \in \Chord(L_{i}, L_{i+j})  
\label{eq:codim_1_stratum_product_moduli_space_CC_diff} 
\end{align}
where in the last type of stratum, $i+j$ is computed modulo $d$,  $x$ agrees with one of the elements of $\vx[1] $, and the sequence obtained by removing $x$  from $\vx[1]$ and replacing it by the sequence $\vx[2]$ agrees with $\vx$ up to cyclic ordering.    Taking the intersection of the last two type of strata with the boundary of  $\Product_{d}^{0}(x_0; \vx) $, we find that $ \partial  \Product_{d}^{0}(x_0; \vx) $  is covered by the closures of codimension $1$ strata:
\begin{align} \label{eq:boundary_moduli_space_f_1}
\Disc(x_0; x)  \times  \Product_{d}^0(x; \vx) &  \quad x  \in \Chord(K, K)  \\ \label{eq:boundary_moduli_space_f_2}
\Product_{d_1}^{0}(x_0; \vx_{1}) \times  \Disc_{d-d_1+1}(x; \vx_2)  & \quad x \in \Chord(L_{i}, L_{i+j}).
\end{align}

Although $\Product_{d}^0(x_0; \vx)  $ is not a product, Equations \eqref{eq:orientation_higher_product} and \eqref{eq:orientation_coproduct} along with our preferred orientations of $\Discbar_{r|1|s}$ and $\Discbar_{d-r-s+1}$ imply that every rigid pair $(u_1,u_2) \in \Product_{d}^0(x_0; \vx) $ induces a map
\begin{equation*}
 o_{x_d} \otimes \cdots \otimes o_{x_1} \to o_{x_0}.
\end{equation*}
Writing $f_{(u_1,u_2)}$ for the map induced on orientation lines, we define
\begin{equation}
  \label{eq:correct_composition_to_glue}
  f = \sum (-1)^{\deg(x_d) + \dagger_{1} + \ddagger_{2} + \circ + \diamond} f_{(u_1,u_2)} 
\end{equation}
where $\dagger_{1}$ is the sign appearing in Equation \eqref{eq:dagger_sign}, applied to $u_1$, and $\ddagger_2$ comes from Equation \eqref{eq:ddagger_sign} applied to $u_2$, and the other two signs come from Equations \eqref{eq:sign_reorder} and \eqref{eq:sign_delta_induce_cyclic}.

\begin{lem}
  $f$ is a chain map of degree $n$, which is homotopic to the composition $\mu \circ CC_{*}(\Delta)$.
\end{lem}
\begin{proof}
The two types of strata appearing in Equations \eqref{eq:boundary_moduli_space_f_1} and \eqref{eq:boundary_moduli_space_f_2} correspond respectively to the compositions $\partial \circ f$ and $f \circ b$ which implies that $f$ is a chain map.   The decomposition of the boundary of  $ \Product_{d}(x_0, \vx) $  given after Lemma \ref{lem:product_moduli_space_manifold} shows that the count of rigid elements of $ \Product_{d}(x_0, \vx) $ defines a homotopy between $f$ and $\mu \circ CC_{*}(\Delta)$.  To check the signs, note that the expression appearing in Equation \eqref{eq:correct_composition_to_glue} reflects the sum of all the signs by which the composition $\mu \circ CC_{*}(\Delta^{r|1|s}) $ differs from the map induced by the product orientation on $ \Product_{d_1}^{0}(x_0; \vx_{1}) \times  \Disc_{d-d_1+1}(x; \vx_2)  $. 
\end{proof}

\subsection{Construction of the second homotopy}

In this section, we define a homotopy between the map $f$ constructed in the previous section, and the composition $\CO \circ \OC$.  We start by making auxiliary choices to perturb the Cauchy-Riemann equation on elements of a moduli spaces of annuli described in Appendix \ref{sec:annuli}:
\begin{defin} \label{def:floer_datum_annnuli}
A \emph{Floer datum} $D_{S}$ on a stable annulus $S \in \Annbar^{-}_{d}$ with $d$ positive boundary punctures $(\xi^1, \ldots, \xi^d) $ and one negative boundary puncture $\xi^{0}$ consists of the following choices on each component:
\begin{enumerate}
\item Strip-like ends near the punctures:  we have  maps $\epsilon^k \co Z_{+} \to S$ for each $1 \leq k \leq d$ converging at the end of the half-strip $Z_{+}$ to $\xi^{k}$, and a map $\epsilon^0 \co Z_{-} \to S  $ converging to  $\xi^{0}$.
\item Time shifting map:  A map $ \rho_{S} \co \partial \bar{S} \to [1,d]$ which is constant near each puncture.  We write $w_{k,S}$ for the value on the $k$\th end and set  
  \begin{equation}
    w_{0,S} = \sum_{k=1}^{d} w_{k,S}.
  \end{equation}
\item Basic $1$-form and Hamiltonian perturbations:  A closed $1$-form $\alpha_{S}$ whose restriction to the boundary vanishes and a map $H_{S} \co S \to \sH(M)$  such that the pullback of $ X_{H_S} \otimes \alpha_{S} $ under $ \epsilon^k $ agrees with  $  X_{\frac{H}{w_{k,S}} \circ \psi^{w_{k,S}}} \otimes dt$.
\item Cutoff $1$-forms: A pair of $1$-forms $\beta_{S}$ and $\gamma_{S}$ which are multiples of $\alpha_{S}$ (by smooth functions), satisfying $d \gamma_{S} \leq 0$ and such that the locus where $d \beta_{S} \geq 0$ is properly contained in the support of $d \gamma_{S}$.   In addition, the pullback of $\gamma_{S}$ under $\epsilon^{0}$ agrees with   
  \begin{equation*}
 \delta_{S} \cdot w_{0,S} \cdot dt 
  \end{equation*}
 for some non-negative constant $\delta_S $, and  $\beta_{S}$ and $\gamma_{S}$  otherwise vanish near the ends.
\item Almost complex structures:  A map $I_{S} \co S \to \sJ(M)$ whose pullback under $ \epsilon^k $ agrees with  $ (\psi^{w_{k,S}})^{*} I_t $ unless $k=0$ in which case it agrees with  $ (\psi^{(1+\delta_S)w_{0,S}})^{*} I_{S^1}$. 
\end{enumerate}
\end{defin}

If $S$ lies on the image of a gluing map (see Equation \eqref{eq:boundary_moduli_annuli_boundary_bubbling}),
\begin{equation}
\Discbar_{d-r-s+1} \times \Discbar_{r|1|s} \to \partial  \Annbar^{-}_{d}
\end{equation}
then we set $\beta_{S} = \gamma_{S} = 0$, and the universal Floer data $\bfD_{f} $  and $\bfD_{\Delta}$ determine the remaining data for $D_{S}$.  

On the other hand, if $S$ lies on the stratum (\ref{eq:boundary_moduli_annuli_interior_breaking}), we set $\gamma_{S}$ to vanish on the component of $S$ lying in  $\Discbar_{d}^{1} $, and use $\bfD_{\OC}$ to define the remaining Floer data.  On the other component (a disc with an interior and a boundary puncture), we use data conformally equivalent to the one fixed in the discussion preceding Equation (\ref{eq:boundary_disc_1_interior_1_boundary_output}).

\begin{defin}
A   \emph{universal and conformally consistent}  choice of Floer data for the second homotopy  is a choice  $\bfD_{\cH^2}$ of Floer data  for every integer $d \geq 1$, and every representative of an element of  $\Annbar^{-}_{d}$ which varies smoothly over the compactified moduli space, such that the two natural Floer data on any irreducible component of a singular annulus are conformally equivalent, and such that near every boundary stratum, $\bfD_{\cH^2}$ agrees to infinite order with the data obtained by gluing.
\end{defin}

Consider a collection of Lagrangians $L_0, L_1, \ldots,  L_d$ with $L_d = L_0$  and a sequence of chords $\vx = \{ x_1, \ldots, x_d \}$ with  $x_k \in \Chord(L_{k-1},L_k)$ for $1 \leq k \leq d$, as well as a chord $x_0 \in \Chord(K,K)$.  We define the moduli space $ \Ann^{-}_{d}(x_0; \vx) $ to be the space of maps $u \co S \to M$ whose source is an arbitrary element of $  \Ann^{-}_{d} $, with asymptotic conditions $  \psi^{w_{k,S}}  \circ  x_k  $ at the $k$\th  incoming ends and $\psi^{(1+\delta_S)w_{0,S}} \circ x_0$ at the outgoing end, with boundary condition $   \psi^{\rho_{S}}L_{k} $ between $\xi^{k-1}$ and $\xi^{k}  $ if $1 \leq k \leq d$, and    $   \psi^{(1+\delta_S)\rho_{S}} K$ on the boundary component containing the outgoing marked point,  and which solve the Cauchy-Riemann equation
\begin{equation}
\label{eq:dbar_equation_homotopy}
\left( du -   X_{H_S}   \otimes \alpha_{S} -  \delta_S \cdot X_{\frac{F \circ \psi^{w_{0,S}}}{w_{0,S}}} \otimes \beta_{S} -   \left( X_{ \frac{H_S}{1+\delta_S} \circ \psi^{1+\delta_S}}  - X_{H_S}  \right) \otimes  \gamma_S  \right)^{0,1} = 0.
\end{equation}
Note that the pullback of Equation \eqref{eq:dbar_equation_homotopy} under the ends agrees with the data used in defining the wrapped Floer complexes, up to composition with $\psi^{w_{k,S}} $ for an incoming end, and  $  \psi^{(1+\delta_S)w_{0,S}}  $ for an outgoing end, so that the asymptotic and boundary conditions make sense.  

\begin{lem}
For generic choices of Floer data $\bfD_{\cH^2}$, the Gromov bordification of $ \Ann^{-}_{d}(x_0; \vx)  $ is a compact manifold of dimension
\begin{equation*}
 \deg(x_0) - n + d  - \sum_{k=1}^{d} \deg(x_k)  
\end{equation*}
which, for a fixed sequence $\vx$ is empty for all but finitely many choices of a chord $x_0$.  Moreover, its boundary decomposes into codimension $1$ strata which are the images of natural inclusions of the moduli spaces
\begin{align}\label{eq:boundary_stratum_annuli_differential_CC_0}
\Product_{d}^{0}(x_0; \vx) & \\
\Discbar^{1}(x_0;y_0) \times \Discbar_{d}^{1}(y_0; \vx) & \quad y_0 \in \Orbit \\
\Discbar(x_0; x)  \times \Annbar_{d}^{-}(x;  \vx) & \quad x \in \Chord(K,K) \\
\Annbar_{d_1}^{-}(x_0;  \vx_{1}) \times  \Discbar_{d-d_1+1}(x; \vx_2)  & \quad 1 \leq d_1 < d+1  \textrm{ and } x \in \Chord(L_{i}, L_{i+j}) \label{eq:boundary_stratum_annuli_differential_CC}
\end{align}
where in the last type of stratum, $\vx_1$ and $\vx_2$ are as in Equation (\ref{eq:codim_1_stratum_product_moduli_space_CC_diff}). 
\end{lem}
\begin{proof}
On every irreducible component of a curve lying in a boundary stratum of the abstract moduli space  $\Annbar^{-}_{d}$, Equation \eqref{eq:dbar_equation_homotopy} agrees (up to applying the Liouville flow for some time) with the Cauchy-Riemann equation coming from the corresponding factor of one of the products appearing in Section \ref{sec:annuli}.  From this, we deduce the Gromov bordification of  $ \Ann^{-}_{d}(x_0; \vx)  $ is given as above.

To prove compactness, we note that $   \left( \frac{H_S}{1+\delta_S} \circ \psi^{1+\delta_S}  - H_S  \right)  d \gamma_S$  agrees with $\delta_S r^{2} d \gamma_{S}$ along the cylindrical end of $M$, and hence dominates $\delta_S \frac{F \circ \psi^{w_{0,S}}}{w_{0,S}} d \beta_{S}$ away from a compact set in $M$ which is independent of the annulus $S$.  In particular, the hypothesis of Lemma \ref{lem:stay_in_compact_set} holds, from which compactness follows.
\end{proof}
 Using the isomorphism \eqref{eq:orientation_annuli}, we find that each rigid element $u \in \Ann^{-}_{d}(x_0; \vx)  $  induces a map $\cH^{2}_{u}$ on orientation spaces, and define
\begin{equation}
  \label{eq:definition_second_homotopy}
  \cH^{2}([x_d], \ldots, [x_1])  = \sum_{\stackrel{\deg(x_0) =n -d  + \sum_{1 \leq k \leq d} \deg(x^k) }{u \in \Disc_{d}^{1}(x_0, \vx) }} (-1)^{\deg(x_d) + \dagger} \cH_{u}^{2}( [x_d], \ldots, [x_1]).
\end{equation}

By comparing the boundary strata listed in Equations \eqref{eq:boundary_stratum_annuli_differential_CC_0}-\eqref{eq:boundary_stratum_annuli_differential_CC} with the formula for the differential defining Hochschild homology, we conclude:
\begin{lem} \label{lem:second_homotopy_sign}
The map $  \cH^{2}$  defines a homotopy between $(-1)^{\frac{n(n+1)}{2}} f$ and the composition $\CO \circ \OC $.  
\end{lem}
\begin{proof}
Again, we only discuss the case $d=1$; the general situation differs from this simplified one by keeping track of the appropriate Koszul signs. Assuming that there is only one input, there is no extraneous sign on $\CO \circ \OC$, and a simple computation shows that over the stratum  $\Discbar^{1}(x_0;y_0) \times \Discbar_{1}^{1}(y_0; x_1)$, the product orientation induces an isomorphism 
\begin{equation}
  \label{eq:interior_breaking_boundary_orientations}
   o_{x_1} \cong o_{y_0} \otimes  o_{y_{-1}}  \otimes  \lambda(L_0) 
\end{equation}
which agrees with the boundary orientation since the tangent vector $\partial_{r}$ to $ \Ann^{-}_{1} $ points outwards at this boundary point. 

Over the boundary stratum $ \Product_{1}^{0}(x_0; x_1) $ where the annulus breaks at chords $y_0$ and $y_{-1}$, the two discs induce isomorphisms
\begin{align*}
 o_{x_1} & \cong o_{y_{-1}} \otimes  o_{y_{0}} \otimes  \lambda^{-1}(L_0)  \\
 o_{x_0} & \cong o_{y_0} \otimes  o_{y_{-1}}.
\end{align*}
 The sign $\diamond$ with value $\deg(y_0) \deg(y_{-1}) $ is the only one which does not vanish in the definition of $f$, which exactly cancels with the Koszul sign to obtain the isomorphism
\begin{equation}
 o_{x_1} \cong o_{y_0} \otimes  o_{y_{-1}}  \otimes  \lambda^{-1}(L_0) .
\end{equation}
In order to arrive at Equation \eqref{eq:interior_breaking_boundary_orientations}, we use the identification $  \lambda(L_0)  \otimes   \lambda(L_0)  \cong \lambda(M) $ and the orientation of $M$ coming from the symplectic form.  This gives an isomorphism of  $\lambda(L_0)$ with its inverse up to the Koszul sign of $(-1)^{\frac{n(n+1)}{2}}$, and we conclude the desired result. 
\end{proof}

\appendix

\section{A universal twisted complex} \label{sec:univ-twist-compl}
We shall prove Lemma \ref{lem:existence_twisted_complex} in this Appendix; the proof is entirely a matter of homological algebra, and the fact that the categories involved originated from geometry will be irrelevant.  The reader should therefore start by consulting Equation \eqref{eq:tensor_over_A} which shows the bar model for $  \cY^{r}_{K} \otimes_{\cB} \cY^{l}_{K}$. Keeping in mind the fact that the differential is obtained by applying the higher products to a collection of successive morphisms, we find that the subspace $ (\cY^{r}_{K} \otimes_{\cB} \cY^{l}_{K})^{N}$  consisting of all expressions involving only less than $N$ factors forms a subcomplex; this is called the length filtration.

The assumption of Lemma \ref{lem:existence_twisted_complex} is that the identity of $K$ lies in the image of $H^{*}(\mu)$.  Since the length filtration is exhaustive, there must therefore be a positive integer $N$ such that a chain-level representative of the identity lies in the image of the evaluation map
\begin{equation}
  \label{eq:multiplication_words_less_than_ell}
   (\cY^{r}_{K} \otimes_{\cB} \cY^{l}_{K})^{N}  \to CW^*(K,K).
\end{equation}

 For each object $L$ of $\Wrap$, let us from now on consider  $\cY^{r}_{L}$ as a module over $\Wrap$.  Assuming for simplicity that the morphism spaces are finite dimensional,  we have a twisted complex
\begin{equation*} \label{eq:universal_module_length_ell}
  \Univ_{K}^{N} = \left( \bigoplus_{k \leq N} CW^{*}(L_k, K)[1] \otimes CW^*(L_{k-1}, L_{k}) [1] \otimes \cdots \otimes  CW^*(L_{0}, L_{1}) [1] \otimes  \cY^{r}_{L_{0}}, D_{U} \right)
\end{equation*}
where the direct sum is taken over all sequences of objects of $\cB$ whose length is less than $N$, and the differential $D_{U}$ is induced by the higher products in $\Wrap$.  The natural maps
\begin{equation} \label{eq:natural_evaluation_yoneda}
CW^{*}(L_{d-1}, L_{d})[1]  \otimes \cdots \otimes CW^*(L_{0}, L_{1}) [1] \otimes  \cY^{r}_{L_{0}} \to  \cY^{r}_{L_{d}}
\end{equation}
induced by composition in $\Wrap$ contribute some of the terms in $D_{U}$, while the remaining ones shorten the sequence $(L_0, \ldots, L_k, K)$  by applying the higher product to a subsequence:
\begin{equation*}
\mu^{d} \co CW^*(L_{i+d-1}, L_{i+d})[1]  \otimes \cdots \otimes CW^*(L_{i}, L_{i+1}) [1]  \to CW^*(L_{i}, L_{d})[1].
\end{equation*}
The reader familiar with the literature on Fukaya categories of Lefschetz fibrations will have encountered this construction of a twisted complex, e.g. in the proof of Proposition 6.2 of \cite{maydanskiy-seidel}.

Note that by replacing $L_d$ by $K$ in Equation \eqref{eq:natural_evaluation_yoneda} we obtain a collection of maps which define a closed morphism
\begin{equation*}
  \Univ_{K}^{N} \to  \cY^{r}_{K}.
\end{equation*}
Composition with this morphism defines a chain map
\begin{equation}
\label{eq:composition_map_modules}
\Hom_{\mod-\Wrap}( \cY^{r}_{K},     \Univ_{K}^{N}  )  \to \Hom_{\mod-\Wrap}( \cY^{r}_{K},   \cY^{r}_{K}  ),
\end{equation}
where $\mod-\Wrap$ is the category of $A_{\infty}$ right  modules over $\Wrap$.

On the other hand, the Yoneda Lemma implies that we have a quasi-isomorphism
\begin{equation*}
  (\cY^{r}_{K} \otimes_{\cB} \cY^{l}_{K})^{N}  \stackrel{\sim}{\to} \Hom_{\mod-\Wrap}( \cY^{r}_{K},     \Univ_{K}^{N}  )
\end{equation*}
which fits into a homotopy commutative diagram
\begin{equation}
  \xymatrix{    (\cY^{r}_{K} \otimes_{\cB} \cY^{l}_{K})^{N}  \ar[r] \ar[d] &  \Hom_{\mod-\Wrap}( \cY^{r}_{K},     \Univ_{K}^{N}) \ar[d] \\
CW^*(K,K) \ar[r] &   \Hom_{\mod-\Wrap}( \cY^{r}_{K},  \cY^{r}_{K}).} 
\end{equation}
where the left vertical arrow comes from Equation \eqref{eq:multiplication_words_less_than_ell}, the right vertical arrow from Equation \eqref{eq:composition_map_modules}, and the bottom horizontal arrow from the Yoneda embedding.  Assuming that a chain representative of the identity in $  HW^*(K,K) $  lies in the image of the left vertical arrow, we conclude that if we pass to the cohomological category $H(\mod-\Wrap)$, the identity of $  \cY^{r}_{K}$ lies in the image of  the composition
\begin{equation*}
  \Hom_{H(\mod-\Wrap)}(   \Univ_{K}^{N},  \cY^{r}_{K}  )  \otimes   \Hom_{H(\mod-\Wrap)}( \cY^{r}_{K},     \Univ_{K}^{N}) \to    \Hom_{H(\mod-\Wrap)}( \cY^{r}_{K},     \cY^{r}_{K} ). 
\end{equation*}
In the cohomological category, we have factored the identity of $  \cY^{r}_{K} $ as the product of a morphism from $ \cY^{r}_{K}$ to $  \Univ_{K}^{N}  $ and one in the opposite direction; this implies that $ \cY^{r}_{K}  $ is a summand of $  \Univ_{K}^{N}  $ in the cohomological category, and hence, since idempotents  lift uniquely up to homotopy to the $A_{\infty}$ refinement (see Section (4b) of \cite{seidel-book}), we conclude that $  \cY^{r}_{K} $  is a summand of $ \Univ_{K}^{N} $  as objects of $   \mod-\Wrap $.  Applying the Yoneda Lemma one last time we conclude that $K$ lies in the subcategory of $\Wrap$ split-generated by the objects of $\cB$.

\begin{proof}[Proof of Lemma \ref{lem:existence_twisted_complex}]  
It remains to remove the finite dimensionality assumption required for $\Univ^{N}_{K}$ to be a twisted complex.  To do this, still using the bar complex model for $ \cY^{r}_{K} \otimes_{\cB} \cY^{l}_{K} $, we choose an arbitrary cycle $\tau$  in $\cY^{r}_{K} \otimes_{\cB} \cY^{l}_{K} $ whose image in $ CW^{*}(K,K)$ represents the identity, and note the existence of a finite dimensional free abelian group $V_{L}$ for each object $L$ of $\Wrap$ such that $\tau $  lies in a subcomplex of the form
\begin{equation}
  \bigoplus_{L} V_{L} \otimes CW^{*}(K,L) \subset  \cY^{r}_{K} \otimes_{\cB} \cY^{l}_{K}.
\end{equation}
With this in mind, we construct a twisted complex $  \bigoplus_{L} V_{L} \otimes  L  $  as a subcomplex of $\Univ_{K}^{N} $ carrying the restriction of $D_{U}$.  Replacing  $\Univ_{K}^{N} $ by this twisted complex in all arguments above, we conclude, as desired, that $K$ lies in the subcategory of $\Wrap$ split-generated by the objects of $\cB$ without having to assume finite dimensionality of morphism spaces.
\end{proof}
\section{Action, energy, and compactness}
Let $W$ be an exact symplectic manifold whose Liouville flow is inward pointing along the  boundary $\partial W$, and such that the complement of a compact subset of $W$ is modelled after the positive end of the symplectisation of a contact manifold.  The only example we shall consider is an end $[R,+\infty) \times \partial M $.  Note that in the general case, there is always a neighbourhood of $\partial W$ which is modelled after $[1,+\infty) \times \partial W$, and  we write $r$  for the radial coordinate on this neighbourhood.

Let $L_{W} \subset W$ be an exact Lagrangian  with a primitive $f_W$ vanishing identically on $\partial L_{W} \subset \partial W$.   In all application, $L_{W}$ will be a union of the ends of objects of $\Wrap$.  

Let $\overline{S}$ be a compact Riemann surface with a decomposition of the boundary
\begin{equation}
  \label{eq:partition_boundary}
  \partial \overline{S} = \partial_{n} \overline{S} \cup \partial_{l} \overline{S}
\end{equation}
and let $S$ be the surface obtained by removing $b_-$ interior marked points and $i_-$ boundary marked points lying on $\partial_{l} \overline{S} $.   Fix negative strip-like or cylindrical ends $\{ \epsilon_{i} \}_{i=1}^{i_-+b_-}$ near all the marked points.

Consider a collection of  $1$-forms $\alpha_k$ on $S$ such that $\alpha_k \wedge \alpha_{\ell} =0$, whose pullbacks under the strip-like ends agree with a positive multiple of $dt$ and whose restrictions to $\partial_{l} S$  vanish, as well as functions  $G_k \co  S \times W \to [0,+\infty) $.  We think of $G_k$ as a function on $W$ parametrised by the points of $S$, and assume that its pullback under a strip-like or cylindrical end depends only on the $t$-coordinate, and that it is independent of $S$ in an open neighbourhood of $\partial W$, where is it given by a multiple of the radial coordinate.  In particular, associated to each end of $S$, there is a function  $H_{i}$ on $[0,1] \times M$ or $S^1 \times M$ such that
\begin{equation*}
 \epsilon_{i}^{*}\left( \sum G_{k} \alpha_{k}\right) = H_{i} \otimes dt.
\end{equation*}

Let us fix for each interior puncture a time-$1$ periodic orbit $y_{i}$ of the Hamiltonian flow of $ H_{i} $, and for each boundary puncture a time-$1$ chord $x_i$ of $H_{i}$; we respectively write $\vx$ and $\vy$  for these sets of chords and orbits.    The action of a chord $x_i$ is given by the expression
\begin{equation} \label{eq:action_chords}
  \Laction(x_i) = - \int_{0}^{1} x_i^{*} \lambda  + \int H_i(x(t)) dt + f_W(x_i(1)) - f_W(x_i(0)).
\end{equation}
while that for an orbit is
\begin{equation} \label{eq:action_orbits}
  \Laction(y_i) = - \int_{0}^{1} y_i^{*} \lambda  + \int H_i(y(t)) dt.
\end{equation}

In addition, choose a family $J_{S}$ of almost complex structures on $W$ parametrised by  $S$ which are of contact type near $\partial W$, and define $\Moduli_{S}(\vx, \vy)$ to be the moduli space of maps
\begin{equation}
  v \co S \to W
\end{equation}
mapping $\partial_{n} S$ to $\partial W$ and $\partial_{l} S$ to $L$, which converge at each end to the appropriate chord or orbit,  and which solve the Cauchy-Riemann equation
\begin{equation}
  \label{eq:generalised_dbar}
  \left(dv - \sum X_{G_k}  \otimes  \alpha_{k} \right)^{0,1}  = 0.
\end{equation}
Note that since we are not imposing Lagrangian boundary conditions along $ \partial_{n} S $,  $\Moduli_{S}(\vx, \vy)$  is not a priori finite dimensional, but this will not affect any of our arguments.

We define two notions of energy for such a map $v$:
\begin{align}
E_{geo}(u) & = \int ||dv - \sum \alpha_{k} \otimes X_{G_k}||^{2} = \int v^{*}(\omega) - \sum v^{*}(dG_k) \wedge \alpha_k \\
E_{top}(u) & =  \int v^{*}(\omega) - d \left( \sum v^{*} G_k  \cdot  \alpha_k \right) = E_{geo} (v) -   \sum  \int  v^*G_k \cdot d \alpha_k. \label{eq:topological_and_geometric}
\end{align}
 
Note that the expression $G_k d \alpha_k$ gives a section on $M \times S$ of the pullback of the space volume forms on $S$.  As $S$ is oriented, it therefore makes sense to say that such a section is positive or negative.  With this in mind, the next result is an essentially minor generalisation of Lemma 7.2 of \cite{abouzaid-seidel}:
\begin{lem} \label{lem:stay_in_compact_set}
If at every point of $S$
\begin{equation} \label{eq:positivity_assumption}
 \sum G_{k} d \alpha_{k} + \sum   \Laction(x_i) + \sum   \Laction(y_i)    \leq 0   
\end{equation}
then every solution to Equation \eqref{eq:generalised_dbar} is a constant map whose image is contained in $\partial L_{W}$. In particular, the moduli space $\Moduli_{S}(\vx, \vy)$ is empty whenever there is at least one puncture.  
\end{lem}
\begin{proof}
This is a direct application of Stokes's theorem, which implies that
\begin{equation*}
E_{top}(v)  =   \int_{\partial_{n}S} (\lambda \circ J_{S}) \circ \left(dv - \sum  X_{G_k} \otimes  \alpha_{k} \right) \circ (- j)  + \sum   \Laction(x_i) +  \sum   \Laction(y_i) 
\end{equation*}
Since $J_{S}$ is assumed to be of contact type near $\partial S$, the first term is negative (see the proof of Lemma 7.2 in \cite{abouzaid-seidel}).  Using Equation \eqref{eq:topological_and_geometric} and the hypothesis \eqref{eq:positivity_assumption}, we conclude that $E_{geo}(v) \leq 0$, which is impossible if $v$ escapes from $\partial W$. 
\end{proof}
In order to apply this result to the various moduli spaces defined in the paper, we need the following action estimates:
\begin{lem}
  \label{lem:action_negative}
Every  orbit in $\Orbit$ lying away from a compact subset of $M$ has negative action.  In addition, if $L$ and $L'$ are objects of $\Wrap$, then every orbit in $\Chord(L,L')$ which lies away from a compact subset also has negative action.
\end{lem}
\begin{proof}
The pullback of $\lambda$ by a Hamiltonian chord agrees with $\lambda(X_{H}) dt$.  Since $H$ is quadratic at infinity, $\lambda(X_H) = 2 r^{2}$ away from a compact subset, which implies that the integral in Equation\eqref{eq:action_chords} is negative if $r$ is sufficiently large.

To prove the analogous result for orbits, we compute that
\begin{align*}
  \cA(y) & = - \int y^{*}(\lambda)  +\int H_{S^1}(y,y(t)) dt \\
& =-  \int \left( \lambda(X_H) + \lambda(X_F) - H -F \right)  \circ y \cdot  dt \\
& = - \int (H - F + \lambda(X_F)) \circ y \cdot dt.
\end{align*}
By Assumption \eqref{eq:support_condition_perturbation}, $H$ dominates $ F$  and $ \lambda(X_F) $ away from a compact subset of $M$, which implies the desired result.
\end{proof}

\section{Moduli spaces and their orientations} \label{sec:abstr-moduli-space}
In this section, we give the necessary ingredients which determine our sign conventions.  We first construct the abstract moduli spaces (and their compactifications) which control the various maps appearing in the paper and fix our conventions for orienting them.  The last section explains how these choices (together with a spin structure on the Lagrangians) determine orientations on the moduli spaces of maps. 

The reader will find bellow listings for the codimension $1$ strata of various moduli spaces.  For those familiar with the operations parametrised by these spaces, it might be helpful to keep in mind the following convention:  the ordering in a decomposition of a stratum into products is listed in the reverse order of the composition of operations to which they correspond.

\subsection{Stasheff polyhedra controlling the $A_\infty$-structure} \label{sec:stash-polyh-infty}

We write $\Discbar_{d}$ for the compactified moduli space of abstract holomorphic discs with $d+1$ boundary punctures of which $d$ are equipped with positive (incoming) ends.  We shall often write a representative $S$ of a point in this moduli space as the complement of $d+1$ marked points $(\xi^0, \ldots, \xi^d)$ in the compactification $\overline{S}$ of $S$.  

Since $\Disc_{d}$ is the quotient of the configuration space of points on $\partial D^2$ by reparametrisations, we orient it by choosing a slice of this action where $(\xi^0, \xi^1, \xi^2)$ are fixed, and using the position of the remaining marked points as local coordinates in the order $(\xi^3, \ldots, \xi^d) $ and with orientation on each factor given by the counterclockwise orientation on the boundary of $D^2$.

\subsection{Stasheff polyhedra controlling the bimodule structure} \label{sec:stash-polyh-contr}
We define $\Disc_{r|1|s}$ to be a copy of the moduli space of discs with $r+s+3$ boundary punctures with the convention that we have equipped $2$ successive punctures $(\xi^{-1}, \xi^{0})$ with negative ends, and the remaining ends (ordered counterclockwise) are denoted $( \xi^{|s}, \ldots, \xi^{|1}, \underline{\xi}, \xi^{1}, \ldots, \xi^{r})  $. 
Each element of this moduli space has a unique representative as $D^2$ with marked points on the boundary in such a way that $(\xi^{-1}, \xi^{0}, \underline{\xi})$ are mapped to the three roots of unity.  By recording the position of the remaining points on the boundary of $D^2$, we obtain coordinates $(z_{|s}, \ldots, z_{|1}, z_{1}, \ldots, z_{r})$ on  $\Disc_{r|1|s}$ which we use to fix the orientation:
\begin{equation*}
  dz_{|s} \wedge \ldots \wedge dz_{|1} \wedge dz_1 \wedge \ldots \wedge dz_{r}.
\end{equation*}

The boundary strata of the Deligne Mumford compactification, which we denote $\Discbar_{r|1|s}$, arise from breakings into a disc with one input and an element of a moduli space of discs with 2 outputs. All such strata  can be obtained as the images of natural inclusions of the following products
\begin{align} 
 \label{eq:boundary_strata_stasheff_bimodules-1} \Discbar_{r-m+1} \times  \Discbar_{m|1|s}  &  \quad 0  \leq m  <  r  \\
\label{eq:boundary_strata_stasheff_bimodules-2} \Discbar_{s-\ell+1} \times \Discbar_{r|1|\ell}  &  \quad  0 \leq \ell <  s  \\
\label{eq:boundary_strata_stasheff_bimodules-3} \Discbar_{r-m|1|s-\ell} \times \Discbar_{\ell+m+1} &  \quad 0 \leq m \leq  r \textrm{ and } 0 \leq \ell \leq s \textrm{ and } (m,\ell) \neq (0,0) \\
\label{eq:boundary_strata_stasheff_bimodules-4} \Discbar_{r|1|s-\ell +k +1} \times \Discbar_{\ell-k} &  \quad 0 \leq  k \leq \ell \leq s \textrm{ and } \ell - k \geq 2 \\
\label{eq:boundary_strata_stasheff_bimodules-5} \Discbar_{r-m +k +1|1|s} \times \Discbar_{m-k} &  \quad 0 \leq k \leq m \leq r \textrm{ and } m - k \geq 2.
\end{align}
The first two moduli spaces correspond to one of the two outputs breaking off, while the last three involve only inputs being involved in the breaking, and are distinguished by  the position of the breaking relative the input $\underline{\xi}$.

\subsection{Discs with one interior puncture} \label{sec:discs-with-one}
We write $\Disc^{1}$ for the one element set corresponding to the unique isomorphism class of discs with one  interior puncture equipped with a positive end, and one boundary puncture equipped with a negative end.

In contrast, we write $\Disc_{d}^1$ for the moduli spaces of  discs with $d$ boundary punctures all of which are equipped with positive ends and $1$ interior puncture equipped with a negative end.  We order the punctures counterclockwise along the boundary, and trivialise $\Disc_{d}^{1}$ by fixing an isomorphism to $D^2$ taking the last boundary puncture to $1$ and the interior puncture to the origin,  and consider the induced orientation
\begin{equation}
 \label{eq:orientation_discs_1_marked_point}  - dz_1 \wedge \cdots \wedge dz_{d-1}
\end{equation}
where $z_1, \ldots, z_{d-1}$ are coordinates which record the positions of the remaining marked points on the boundary $S^1 \subset D^2$, oriented counterclockwise.

The Deligne-Mumford compactification of $\Disc_{d}^1$ can be made into a manifold with corners.  Its codimension $1$ strata are given by choosing a partition $d_1+ d_2 = d +1$ with $2 \leq d_2 \leq d$, and considering the $d$ possible choices of natural inclusions
\begin{equation}
\label{eq:boundary_discs_1_puncture}
\Disc_{d_1}^1 \times   \Disc_{d_2} \to \Discbar_{d}^1.
\end{equation}
These inclusions come in two types:
\begin{enumerate}
\item \label{item:bar_differential_degeneration} For $0 \leq  k  < d_1 -1$ one constructs a nodal disc by identifying the outgoing marked point of an element of $  \Disc_{d_2}  $ with the $k+1$\st marked point of an element of $  \Disc_{d_1}^1$.  The boundary marked points of the nodal discs are ordered in the unique way which is (1) compatible with the orientation of the boundary, and (2) such that the last marked point comes from the last marked point of $  \Disc_{d_1}^1$.     
\item \label{item:cyclind_differential_degeneration} For $0 \leq k \leq d_2-1 $  one constructs a nodal disc by identifying the outgoing marked point of an element of $  \Disc_{d_2}  $ with the last marked point of an element of $ \Disc_{d_1}^1$ and choosing the $k+1$\st marked point coming from  $  \Disc_{d_2}  $  to be the terminal element of the marked points on the nodal disc.
\end{enumerate}

\begin{rem}
If we were to develop a more complete theory of algebraic invariants on symplectic cohomology, we would have to consider moduli spaces of discs with $1$ interior puncture together with the datum of an asymptotic marker at this marked point.  However, it is sufficient for the purposes of this paper to fix a marker which ``points away'' from the last marked point on the boundary.  This idea is implemented in our fixed choice of cylindrical ends in Definition \ref{def:floer_datum_discs_interior_puncture}.
\end{rem}

\subsection{Annuli}\label{sec:annuli}
We write $\Ann_{d}$ for the abstract moduli space of holomorphic annuli with $d$ boundary punctures lying on one boundary component and labelled as incoming and a unique outgoing boundary puncture lying on the other component.  For each positive real number $r$, the domain
\begin{equation}
  \label{eq:explicit_annulus}
  \{ z \in \bC | 1 \leq |z| \leq r \}
\end{equation}
with the pair of marked points $\{ 1,-r \}$ gives an element of  $\Ann_{1}$.  We write  $\Ann^{-}_{1} $ for the real line consisting of such annuli, and 
\begin{equation}
  \label{eq:restricted_annuli}
  \Ann^{-}_{d} \subset \Ann_{d}
\end{equation}
for the inverse image of  $\Ann^{-}_{1} $  under the map $\Ann_{d} \to \Ann_1  $ which forgets all the incoming marked points but the last.  We fix an orientation for  $\Ann^{-}_{d}$ 
\begin{equation}
  \label{eq:orientation_moduli_annuli}
  dr \wedge dz_1 \wedge \cdots \wedge dz_{d-1}
\end{equation}
where $z_{j}$, as in Equation \eqref{eq:orientation_discs_1_marked_point}, is a coordinate recording the position of the $j$\th incoming marked point.

We have a Deligne-Mumford compactification $\Annbar^{-}_{d}$ which is a manifold with boundary.  To understand the structure at the boundary, we first note that  $\Annbar^{-}_{1}$ may be identified with the interval $[0,+\infty]$ by using the modular parameter $e^{r}$ as coordinate.  At $+\infty$, the annulus breaks into two discs each carrying both an interior and a boundary puncture; we shall label as outgoing the interior puncture on the disc carrying the incoming boundary marked point, so that this stratum may be identified with the product $\Disc^{1} \times \Disc_{1}^{1}$.

At the other boundary point of $\Annbar^{-}_{1}$, the annulus breaks into two discs meeting at two nodes.  By convention, the nodes give rise to two outgoing marked points on the disc carrying the incoming marked point which survives gluing, so that we may identify this boundary stratum with $\Disc_{0|1|0} \times \Disc_{2}$.

\begin{rem}
While the appearance of two boundary nodes is a codimension $2$ phenomenon in the Deligne-Mumford compactification of $\Ann_{d}$, the inclusion of $ \Ann^{-}_{1}$ in $\Annbar_{1}$ meets this codimension $2$ stratum cleanly but not transversely, so that the resulting stratum is of codimension $1$ in  $\Annbar^{-}_{d}$.
\end{rem}

Forgetting all but the last incoming marked point (and collapsing any unstable component) defines a map $\Annbar^{-}_{d} \to \Annbar^{-}_{1}$.  The fibre over $ \Disc^{1} \times \Disc_{1}^{1} $ is  a boundary stratum carrying a natural identification with
\begin{equation}
\label{eq:boundary_moduli_annuli_interior_breaking}
\Disc^{1} \times \Discbar_{d}^{1}  \end{equation}
while the fibre over $\Disc_{0|1|0} \times \Disc_{2}$ is a union over all non-negative integers $r$ and $s$ such that $r+s \leq d$ of the images of inclusion maps
\begin{equation}
\label{eq:boundary_moduli_annuli_coproduct_relation}
\Discbar_{d-r-s+1} \times \Discbar_{r|1|s} \to \partial \Annbar^{-}_{d}.
\end{equation}

The remaining codimension $1$ strata of  $\Annbar^{-}_{d}$ submerse over the interior of $\Annbar^{-}_{1}  $ and are obtained as the union over all integers $1 \leq d_1 < d$ of product
\begin{equation}
 \label{eq:boundary_moduli_annuli_boundary_bubbling}
\Annbar^{-}_{d_1} \times \Discbar_{d-d_1 +1}.
\end{equation}
More precisely, among the boundary strata of $\Annbar^{-}_{d}$  there are $d_1$ different ones each naturally identified with the above product, corresponding geometrically to forming a nodal curve attaching the outgoing marked point of a disc with $d-d_1+ 1$ marked points at any of the incoming marked points of an annulus.

\subsection{An auxiliary moduli space}\label{sec:an-auxiliary-moduli}
For each integer $d \geq 1$, let 
\begin{equation} \label{eq:disjoint_union_manifolds}
  \Product_{d} = [0, 1] \times \bigcup_{r+s+d_1 - 1 =  d} \Discbar_{d-r-s+1} \times \Discbar_{r|1|s} / \sim
\end{equation}
where  $\sim$ is the equivalence relation which identifies points which have the same image under the maps  \eqref{eq:boundary_moduli_annuli_coproduct_relation}.  For an explicit description, note that  for each integer $\ell \leq s$,   we have a boundary stratum of $ \Discbar_{d-r-s+1} \times \Discbar_{r|1|s} $  which is a triple product
\begin{equation}
 \Discbar_{d-r-s+1} \times \Discbar_{s-\ell+1}  \times  \Discbar_{r|1|\ell}.
\end{equation}
coming from the product of the first factor with the boundary stratum of $\Discbar_{r|1|s} $ given in Equation \eqref{eq:boundary_strata_stasheff_bimodules-2}.  Note that this stratum may also be naturally identified with a boundary stratum of the product $ \Discbar_{d-r-\ell+1} \times \Discbar_{r|1|\ell}$.  Similarly,  $  \Discbar_{d-r-s+1} \times \Discbar_{r|1|s} $  and  $  \Discbar_{d-k-s+1} \times \Discbar_{k|1|s} $ have boundary strata which may be naturally identified.    These identifications generate the equivalence relation $\sim$.

We write
\begin{equation}
  \label{eq:boundary_interval_strata}
  \Product_{d}^{0} \textrm{ and }   \Product_{d}^{1} 
\end{equation}
for the boundary strata of $ \Product_{d}$  coming from the boundary of the interval $[0,1]$.  Their complement in  the boundary of $ \Product_{d}$  is covered by the images of natural inclusions
\begin{equation}
  \label{eq:boundary_product_disc_bubbling}
  \Product_{d_1} \times \Discbar_{d - d_1 +1} \to \partial  \Product_{d}
\end{equation}
for every integer $1 \leq d_1 <  d $ and $1 \leq k \leq d_1 $ obtained by attaching a disc with $d-d_1+1 $ inputs at the $k$\th input of an element $\Product_{d_1}$.

\subsection{Orienting the moduli spaces of maps}
Given $y \in \Orbit$, there is a unique homotopy class of trivialisations of the pullback of $TM$ to $S^1$ which is compatible with the chosen trivialisation of $\Lambda^{n}_{\bC} T^{*}M$.  Given a cylindrical end $[1,+\infty) \times S^1$,  the linearisation of Equation \eqref{eq:dbar_strip} with respect to such a trivialisation exponentially converges at the end to an operator of the form
\begin{equation*} Y \mapsto \partial_{s} Y - J_{t} \partial_{t} Y -   A(+\infty, t) Y   \end{equation*}
such that $  J_{t} \partial_{t} -   A(+\infty, t)  $ is self-adjoint (see \cite{floer-hofer}).

Let us therefore consider the plane $\bC$ equipped with a negative cylindrical end
\begin{align*}
  \epsilon \co Z_{-}  & \to \bC  \\
(s,t) & \mapsto e^{-s+ 2 \pi \sqrt{-1} t}
\end{align*}
and extensions of $J_t$ and $A(-\infty, t) $ to a family of (linear) complex structures on $\bC^{n}$ and endomorphisms of $\bC^{n}$ parametrised by $\bC$.  In particular, we obtain an operator
\begin{equation}
D_{y} \co W^{1,2}(\bC, \bC^{n}) \to L^{2} (\bC, \bC^{n}) .
\end{equation}

\begin{defin} \label{def:orientation}
The \emph{orientation line}  $o_{y}$ is the determinant line $\det(D_y)$.  
\end{defin}

For a time-$1$ chord $x$ with endpoints on graded Lagrangian $L_i$ and $L_j$, there is an analogously defined orientation line $o_{x}$.  The reader is referred to Section (11l) of \cite{seidel-book} for the definitions.

There is a general theory, starting with \cite{floer-hofer}, studied in the Lagrangian case in \cite{FOOO}, and phrased in \cite{seidel-book} using precisely the language adopted here, for relating such orientation lines with orientations of moduli spaces of punctured holomorphic curves with asymptotic conditions converging to the selected  Hamiltonian chords and orbits.  Consider a moduli space $\Moduli$ of compact Riemann surfaces with boundary $\overline{S}$ with  $i_- + i_+$ interior marked point and $b_- + b_+$ boundary marked points:  $\pm$ represents the respective numbers of marked points equipped with positive or negative ends.  Let $S$ denote the complement of the marked points, and fix a smooth map from $\partial S$ to the space of exact graded Lagrangians on $M$ which is locally constant near the ends.  In particular, we have two graded Lagrangians associated to each boundary marked point in $\overline{S}$  coming from the two local components of $\partial S$. 

Let $\{ H_{k}^{-} \}_{k=1}^{i_-+b_-} $ and  $\{ H_{k}^{+} \}_{k=1}^{i_+ + b_+} $ be periodic time-dependent  Hamiltonians, and consider collections $\vy[+] = \{ y_i^{+} \}_{i=1}^{i_+}$  and  $ \vy[-] = \{ y_i^{-} \}_{i=1}^{i_-}  $ of time-$1$ orbits for the Hamiltonians $H^{+}_{i}$ and $H_{i}^{-}$  as well as time-$1$ orbits $\vx[+]=\{ x_i^{+} \}_{i=i_+ + 1}^{b_+}$ and  $ \vx[-] =  \{ x_i^{-} \}_{i=i_-+1}^{b_-}  $ with endpoints on the two Lagrangians associated to each boundary marked point.

Choosing a family $I_{S}$ of almost complex structures on $M$ parametrised by points in $S$, we define $\Moduli( \vx[-], \vy[-];  \vx[+] , \vy[+])   $  to be the moduli space of maps with appropriate Lagrangian boundary conditions and asymptotic limits, which solve a family of equations
\begin{equation} \dbar_{S} \equiv \left(du - Y_S\right)^{0,1} = 0\end{equation}
for each surface $S \in \Moduli$ where $Y_S$ is a $1$-form on $S$ valued in the space of Hamiltonian vector fields on $M$ whose pullback under an appropriate strip-like end near the punctures agrees with $X_{H_i^{-}} \otimes dt$ near the negative punctures and $X_{H_i^+} \otimes dt$  near the positive ones.  

The main result we need is:
\begin{lem} \label{lem:guts_orientations}
Let $\bar{C}_1, \ldots, \bar{C}_k$ be the components of the boundary of $\bar{S}$, and $e_j$ denote the number of negative ends of $\bar{C}_{j}$.  If we fix a marked point  $z_j \in C_j $ mapping to a Lagrangian $L_j$ then, assuming the moduli space  $\Moduli( \vx, \vy;  \vx[+] , \vy[+])   $   is regular at a point $u$, we have a canonical isomorphism
\begin{equation*} \lambda \left( \Moduli( \vx[-], \vy[-]; \vx[+], \vy[+] ) \right) \cong  \lambda( \Moduli ) \otimes  \bigotimes_{j}  \left(  \lambda \left(T|_{u(z_j)} L_j \right)^{\otimes 1-e_j}  \otimes  \bigotimes_{x_i^{\pm} \in \bar{C}_{j} } o_{x_i^{\pm}}^{\mp}   \right)   \otimes \bigotimes_{y_i^{\pm}} o_{y_i^{\pm}}^{\mp}    \end{equation*}
where $\lambda$ stands for the top exterior power of the tangent space.
\end{lem}
\begin{proof}
This  a minor generalisation of Proposition 11.13 in \cite{seidel-book} which is stated in the absence of interior punctures, and for a moduli space of solutions for a fixed conformal structure.  Writing $\tilde{S}$ for the surface obtained by filling  interior punctures, we obtained an operator $D_{\dbar_{\tilde{S}}}$ by gluing the linearisation $D_{\dbar_{S}}$ of $\dbar_{S}$ to the various operators $D_{y}$.  The gluing theorem implies that we have an isomorphism
\begin{equation*} \det( D_{ \dbar_{\tilde{S}}}) \cong  \det( D_{\dbar_{S}})  \otimes  \bigotimes_{y_i^{\pm}} o_{y_i^{\pm}}^{\pm}  . \end{equation*}
The result now follows directly by applying Proposition 11.13 in \cite{seidel-book}.
\end{proof}

Applying this result to the various moduli spaces appearing in this paper, and using the Koszul rules to reorder factors at our convenience, we obtain:

\begin{cor}
Given a pair of orbits  $y_0, y_1 \in \Orbit$, a sequence $\vx = ( x_1,  \ldots, x_d  )$ of chords with $x_k$ starting at $L_{k-1}$ and ending at $L_{k}$,  and a chord  $x_0$ with endpoints on a Lagrangian $K$,  there are canonical up to homotopy isomorphisms
\begin{align}
\label{eq:orientation_differential_SH} \lambda( \Cylt(y_0;y_1) )  \otimes  o_{y_1} & \cong  o_{y_0} \\
\label{eq:orientation_differential_HW}  \lambda( \Disct(x_0;x_1) )  \otimes  o_{x_1} & \cong  o_{x_0} \\ 
\label{eq:orientation_SH_to_HW}  \lambda( \Disc^{1}(x_0;y_1) )  \otimes  o_{y_1}  & \cong  o_{x_0}   \\
\label{eq:orientation_higher_product}\lambda( \Disc_{d}(x_0;\vx) )  \otimes   o_{x_d}  \otimes \cdots \otimes o_{x_2} \otimes o_{x_1}  & \cong  \lambda(\Disc_{d}) \otimes o_{x_0} \\
\label{eq:orientation_HH_to_SH}  \lambda( \Disc^1_{d}(y_0;\vx) )   \otimes    o_{x_d}  \otimes \cdots  \otimes o_{x_2} \otimes o_{x_1}  \otimes  \lambda^{-1}(L_0)  & \cong  \lambda(\Disc_{d}^1)\otimes o_{y_0}  \\
\label{eq:orientation_annuli} \lambda( \Ann_{d}(x_0;\vx) )   \otimes o_{x_d} \otimes \cdots \otimes o_{x_2}  \otimes o_{x_1}   \otimes  \lambda^{-1}(L_0)    & \cong   \lambda(\Ann_{d}) \otimes  o_{x_0} 
\end{align}
In addition, if $x_{-1}$,  $ x_0$, and $\underline{x}$ are respectively chords in $\Chord(K, L_r) $,  $\Chord(L_{|s}, K) $, and  $\Chord(L_{|0}, L_0) $, while  $\vx = ( x_{1},  \ldots, x_r )$ and  $\vx[|] = (x_{|s}, \ldots, x_{|1} )$ are sequences of chords respectively in $ \Chord(L_{|k+1}, L_{|k}) $  and $ \Chord(L_{k}, L_{k+1}) $, we have an isomorphism
\begin{multline}
   \lambda( \Disc_{r|1|s}(x_{-1},x_0;\vx[|], \underline{x}, \vx[r])  \otimes o_{x_{r}} \otimes \cdots \otimes  o_{x_{1}} \otimes o_{\underline{x}} \otimes o_{x_{|1}} \otimes \cdots \otimes o_{x_{|s}} \otimes \lambda(L_0) \\   \cong  \lambda(\Disc_{r|1|s}) \otimes  o_{x_{-1}} \otimes  o_{x_{0}}  \label{eq:orientation_coproduct} 
\end{multline}\qed
\end{cor}
\begin{rem} \label{rem:trivialise_R_action}
To obtain operations from Equations \eqref{eq:orientation_differential_SH} and \eqref{eq:orientation_differential_HW}, one must in addition choose a trivialisation of the $\bR$ action on the moduli spaces $\Cylt (y_0;y_1)$ and $ \Disct(x_0;x_1)$.  In both cases, we choose $\partial_{s}$ as the vector field inducing the trivialisation. 
\end{rem}


\begin{bibdiv}
\begin{biblist}


\bib{abouzaid-seidel}{article}{
   author={Abouzaid, Mohammed},
   author={Seidel, Paul},
   title={An open string analogue of Viterbo functoriality},
   journal={Geom. Topol.},
   volume={14},
   date={2010},
   number={2},
   pages={627--718},
   issn={1465-3060},
   review={\MR{2602848}},
   doi={10.2140/gt.2010.14.627},
}

\bib{abouzaid-cotangent}{article}{
  author = {Mohammed Abouzaid},
  title = {A cotangent fibre generates the Fukaya category},
  eprint = {arXiv:1003.4449},
}

\bib{abouzaid-htpy}{article}{
  author = {Mohammed Abouzaid},
  title = {Maslov 0 nearby Lagrangians are homotopy equivalent},
  eprint = {arXiv:1005.0358},
}

\bib{beilinson}{article}{
   author={Be{\u\i}linson, A. A.},
   title={Coherent sheaves on ${\bf P}^{n}$ and problems in linear
   algebra},
   language={Russian},
   journal={Funktsional. Anal. i Prilozhen.},
   volume={12},
   date={1978},
   number={3},
   pages={68--69},
   issn={0374-1990},
   review={\MR{509388 (80c:14010b)}},
}


\bib{BEE}{article}{
  author = {Frederic Bourgeois and Tobias Ekholm and Yakov Eliashberg},
  title = {Effect of Legendrian Surgery},
  eprint = {arXiv:0911.0026},
}

\bib{costello}{article}{
   author={Costello, Kevin},
   title={Topological conformal field theories and Calabi-Yau categories},
   journal={Adv. Math.},
   volume={210},
   date={2007},
   number={1},
   pages={165--214},
   issn={0001-8708},
   review={\MR{2298823 (2008f:14071)}},
   doi={10.1016/j.aim.2006.06.004},
}

\bib{floer}{article}{
   author={Floer, Andreas},
   title={Morse theory for Lagrangian intersections},
   journal={J. Differential Geom.},
   volume={28},
   date={1988},
   number={3},
   pages={513--547},
   issn={0022-040X},
   review={\MR{965228 (90f:58058)}},
}


\bib{floer-hofer}{article}{
   author={Floer, A.},
   author={Hofer, H.},
   title={Coherent orientations for periodic orbit problems in symplectic
   geometry},
   journal={Math. Z.},
   volume={212},
   date={1993},
   number={1},
   pages={13--38},
   issn={0025-5874},
   review={\MR{1200162 (94m:58036)}},
   doi={10.1007/BF02571639},
}



\bib{FHS}{article}{
    author={Floer, Andreas},
    author={Hofer, Helmut},
    author={Salamon, Dietmar},
     title={Transversality in elliptic Morse theory for the symplectic
            action},
   journal={Duke Math. J.},
    volume={80},
      date={1995},
    number={1},
     pages={251\ndash 292},
      issn={0012-7094},
    review={MR1360618 (96h:58024)},
}

\bib{FOOO}{book}{
   author={Fukaya, Kenji},
   author={Oh, Yong-Geun},
   author={Ohta, Hiroshi},
   author={Ono, Kaoru},
   title={Lagrangian intersection Floer theory: anomaly and obstruction.
   Part I},
   series={AMS/IP Studies in Advanced Mathematics},
   volume={46},
   publisher={American Mathematical Society},
   place={Providence, RI},
   date={2009},
   pages={xii+396},
   isbn={978-0-8218-4836-4},
   review={\MR{2553465}},
}

\bib{FSS}{article}{   
author={Fukaya, Kenji},
   author={Seidel, Paul},
   author={Smith, Ivan},
title={The Symplectic Geometry of Cotangent Bundles from a Categorical Viewpoint},
book={ series={	Lecture Notes in Physics }, 
volume =  {  757},
publisher = {Springer},
place= {  Berlin / Heidelberg}, },
date = {2009},
pages ={1--26},
}

\bib{KS}{article}{
   author={Kontsevich, M.},
   author={Soibelman, Y.},
   title={Notes on $A_\infty$-algebras, $A_\infty$-categories and
   non-commutative geometry},
   conference={
      title={Homological mirror symmetry},
   },
   book={
      series={Lecture Notes in Phys.},
      volume={757},
      publisher={Springer},
      place={Berlin},
   },
   date={2009},
   pages={153--219},
   review={\MR{2596638}},
}

\bib{MWW}{article}{ 
author =  {S. Mau},
author = {K. Wehrheim}, 
author={C. Woodward},
title={$A_{\infty}$ functors for Lagrangian correspondences},
status={In preparation},
 date ={2010},
}

\bib{maydanskiy-seidel}{article}{ 
   author={Maydanskiy, Maksim},
   author={Seidel, Paul},
   title={Lefschetz fibrations and exotic symplectic structures on cotangent
   bundles of spheres},
   journal={J. Topol.},
   volume={3},
   date={2010},
   number={1},
   pages={157--180},
   issn={1753-8416},
   review={\MR{2608480}},
   doi={10.1112/jtopol/jtq003},
}

\bib{seidel-GL}{article}{
   author={Seidel, Paul},
   title={Graded Lagrangian submanifolds},
   language={English, with English and French summaries},
   journal={Bull. Soc. Math. France},
   volume={128},
   date={2000},
   number={1},
   pages={103--149},
   issn={0037-9484},
   review={\MR{1765826 (2001c:53114)}},
}

\bib{seidel-bimodules}{article}{
   author={Seidel, Paul},
   title={$A_\infty$-subalgebras and natural transformations},
   journal={Homology, Homotopy Appl.},
   volume={10},
   date={2008},
   number={2},
   pages={83--114},
   issn={1532-0073},
   review={\MR{2426130}},
}

\bib{seidel-book}{book}{
   author={Seidel, Paul},
   title={Fukaya categories and Picard-Lefschetz theory},
   series={Zurich Lectures in Advanced Mathematics},
   publisher={European Mathematical Society (EMS), Z\"urich},
   date={2008},
   pages={viii+326},
   isbn={978-3-03719-063-0},
   review={\MR{2441780}},
}

\bib{viterbo}{article}{
      author={Viterbo, C.},
       title={Functors and computations in {F}loer homology with applications,
  {P}art {I}},
        date={1999},
     journal={Geom. Funct. Anal.},
      volume={9},
       pages={985\ndash 1033},
}

\end{biblist}
\end{bibdiv}
\end{document}